\documentclass[11pt]{article}

\usepackage[margin=1in]{geometry}
\usepackage[T1]{fontenc}
\usepackage[utf8]{inputenc}
\usepackage{lmodern}
\usepackage{microtype}
\usepackage{amsmath,amssymb,amsthm,mathtools}
\usepackage{booktabs,array}
\usepackage{enumitem}
\usepackage{xcolor}
\usepackage[hidelinks]{hyperref}
\usepackage{doi}
\usepackage{xurl}
\usepackage[backend=biber,citestyle=numeric-comp,bibstyle=ieee,sorting=none,minbibnames=5,maxbibnames=8,giveninits=true]{biblatex}
\AtBeginBibliography{\footnotesize\setlength{\bibitemsep}{0.25\baselineskip}}
\DeclareFieldFormat{doi}{\href{https://doi.org/#1}{\nolinkurl{doi:#1}}}
\hypersetup{
  pdftitle={Reciprocal Dynamics in Polynomial Congruence Jumping: Integral Orbits and Valuation Cones},
  pdfauthor={Max A. Alekseyev and Dmitry I. Khomovsky},
  pdfsubject={Reciprocal dynamics in congruence jumping, integral orbits, p-adic valuation cones, C-finite point dynamics, hyperbolic growth obstructions, and Vieta dynamics},
  pdfkeywords={discrete dynamical systems, reciprocal polynomial maps, arithmetic dynamics, p-adic valuations, invariant cones, C-finite sequences, hyperbolic growth, Vieta dynamics}
}
\newtheorem{theorem}{Theorem}[section]
\newtheorem{proposition}[theorem]{Proposition}
\newtheorem{lemma}[theorem]{Lemma}
\newtheorem{corollary}[theorem]{Corollary}

\newtheorem{example}[theorem]{Example}
\theoremstyle{definition}
\newtheorem{definition}[theorem]{Definition}
\theoremstyle{remark}
\newtheorem{remark}[theorem]{Remark}

\newcommand{\defeq}{:=}
\newcommand{\Z}{\mathbb Z}
\newcommand{\Q}{\mathbb Q}

\newcommand{\Rcal}{\mathcal R}
\newcommand{\Mcal}{\mathcal M}

\DeclareMathOperator{\cont}{cont}
\title{Reciprocal Dynamics in Polynomial Congruence Jumping:\\Integral Orbits and Valuation Cones}
\author{Max A. Alekseyev\thanks{The George Washington University, Washington, DC, USA. Corresponding author. Email: \href{mailto:maxal@gwu.edu}{maxal@gwu.edu}. ORCID: \href{https://orcid.org/0000-0002-5140-8095}{0000-0002-5140-8095}}
\and
Dmitry I. Khomovsky\thanks{Independent researcher. Email: \href{mailto:dmkhomovsky@gmail.com}{dmkhomovsky@gmail.com}. ORCID: \href{https://orcid.org/0009-0000-8478-6220}{0009-0000-8478-6220}}}
\date{}
\begin{document}
\maketitle

\begin{abstract}
Alternating reciprocal transformations of polynomial divisibility systems couple an orbit of integer points to a rational dynamical system on polynomial coefficients. For fixed positive degrees $n,m$, we prove universal four-step return to rational gauge classes. Weighted primitive representatives preserve integral continuation through nondegenerate jumps using only exact division and coefficient gcds; the stronger requirement of monic integral representatives is governed by endpoint $p$-adic valuation defects.

After one integral block, constant-term valuations evolve by an explicit unimodular matrix whose trace depends only on $nm$. This gives finite-order behavior for $nm<4$, parabolic dilations for $nm=4$, and hyperbolic dynamics for $nm>4$. In the hyperbolic case, invariant thresholds, irrational eigencones, and discrete valley formulas give finite criteria for forward and bi-infinite integral coefficient orbits, although failure can be delayed arbitrarily long.

Every positive-constant canonical monic coefficient chain admits an explicit positive rational four-phase gauge. At $nm=4$, geometric gauges reduce nondegenerate compatible rational point orbits to two-periodic affine recurrences, hence to sequences satisfying linear recurrences with constant coefficients. Polynomial identities of degree at most $h$ on each invariant conic admit certificates from $2h+1$ initial phase states, while integral unimodular recurrence models yield periodic modular obstructions. For $nm>4$, eventually positive normalized point orbits tending to infinity have hyperbolic logarithmic growth and cannot satisfy such linear recurrences. Finally, separated auxiliary-relation transport reduces rational rank to at most one after two alternating jumps; quadratic specialization gives exact local saturation criteria and persistence results across periodic, parabolic, and hyperbolic orbits.
\end{abstract}



\section{Introduction}

The starting point is the familiar idea of \emph{jumping}: replace one coordinate of an integer solution by another integer obtained from the same equation or divisibility relation, and iterate.  In classical Vieta jumping the equation itself is fixed.  Our motivating setting is the polynomial divisibility system
\begin{equation}\label{eq:intro-divisibility-system}
  y\mid A(x),\qquad x\mid B(y),
\end{equation}
with $A,B\in\mathbb Z[T]$.  If
\[
  y'\defeq \frac{A(x)}{y},
\]
then $y'$ is again integral.  The second divisibility relation, however, need not retain the same polynomial $B$: after substituting the quotient move, its coefficients transform by a reciprocal-polynomial operation.  Iteration therefore produces two coupled evolutions, an arithmetic orbit of integer points and a rational orbit of polynomial coefficient states.

This mechanism was introduced and developed from the arithmetic point of view in our companion paper on congruence jumping~\cite{AlekseyevKhomovsky2026Arithmetic}.  That paper studies reciprocal transport of the divisibility conditions, primitive normalization, pointwise lifting, and the resulting solution sets.  The present paper separates the second evolution and studies the polynomial coefficients as a dynamical system in their own right.  For fixed positive degrees $n$ and $m$, a polynomial pair is a point of a finite-dimensional coefficient space, and alternating normalized reciprocal transformations act rationally on the open set where the relevant constant terms are nonzero.  This viewpoint exposes qualitative structure that is largely independent of the original Diophantine problem.

Iteration of rational maps under arithmetic restrictions is a central theme of arithmetic dynamics~\cite{Silverman2007}, while the linear recurrences that appear after valuation reduction belong to the classical theory of difference equations~\cite{Elaydi2005}.  The present system has an unusual combination of the two structures.  At the coefficient level, four alternating reciprocal updates close up to rational gauge: every state returns to its original rational gauge classes.  At the arithmetic level, requiring all intermediate coefficients to remain integral produces a nontrivial local dynamical constraint.  The constant-term valuations then form a two-dimensional linear factor of the coefficient dynamics, with invariant cones and finite valley tests.

Classical fixed-state quotient dynamics already appears in the quadratic chains of Mills~\cite{Mills1953,Mills1954}.  Birational recurrences with unexpectedly persistent integrality also arise in the Laurent-phenomenon and cluster settings: the Laurent phenomenon produces Laurent iterates under broad mutation rules~\cite{FominZelevinsky2002}, mutation-periodic quivers generate finite-dimensional birational cluster maps~\cite{FordyHone2014}, and Vieta-jumping transformations for Markov-type Diophantine equations can be realized by cluster and generalized-cluster mutations~\cite{Lampe2016,GyodaMatsushita2023}.  A related reciprocal-substitution pattern occurs in rank-two Laurent phenomenon algebras~\cite{LamPylyavskyy2016}.  Our emphasis is different: the polynomial coefficients themselves form the state variables, while integrality is governed by the arithmetic of reciprocal normalization rather than by a Laurent property.

The classification has three layers.  First, endpoint valuation defects provide an exact local criterion for one monic reciprocal normalization and hence for one complete four-step block.  Second, after the first block both states are monic, and the constant-term valuations evolve by an explicit matrix $M_{n,m}\in\operatorname{SL}_2(\mathbb Z)$ depending only on the degree product $s=nm$.  For $s=1,3$ this matrix is elliptic of finite order, for $s=2$ it is the central matrix $-I$, for $s=4$ it is parabolic, and for $s>4$ it is hyperbolic.  Thus $s<4$ is a finite-order, elliptic-type regime, whereas $s=4$ and $s>4$ exhibit genuinely parabolic and hyperbolic behavior.  At the parabolic boundary, the coefficient dilation also induces a natural geometric gauge on compatible point orbits; for all three degree patterns with $nm=4$ the normalized point dynamics becomes C-finite.  For every positive-constant canonical monic chain we construct a positive rational four-phase coefficient gauge. For $nm>4$, every eventually positive normalized orbit that tends to infinity then has superexponential point growth and cannot be C-finite.  Third, in the hyperbolic regime the full coefficient-integrality problem reduces to invariant endpoint thresholds and minima of four second-order recurrence subsequences, giving finite forward and two-sided decision procedures despite the existence of arbitrarily delayed obstructions.

Coefficient integrality and pointwise liftability are logically distinct.  A coefficient state may be integral even when no prescribed integer point realizes the next primitive normalization, while pointwise identities can validate cancellations not visible from coefficient content alone.  Pointwise lifting and the arithmetic structure of the solution sets are the principal concerns of the companion paper~\cite{AlekseyevKhomovsky2026Arithmetic}; Sections~\ref{sec:four-step-block}--\ref{sec:iterated-gauge} here classify the intrinsic coefficient system. Section~\ref{sec:point-dynamics} studies compatible point recurrences, and Section~\ref{sec:quadratic-coeff} returns to primitive quadratic coefficient states.  Section~\ref{sec:persistent-saturation} then returns to the pointwise interface.  It first shows that simultaneous separated auxiliary relations have a strict rational-rank bottleneck: transport factors through a side projection, and two alternating jumps leave at most one rational relation direction even after integral saturation.  We then specialize to quadratic Vieta dynamics, where strengthened auxiliary relations can persist indefinitely on a periodic orbit and for arbitrarily long finite windows on parabolic and hyperbolic orbits.  The longer constructions are deferred to Appendix~\ref{app:persistence-proofs}.

The organization follows this decomposition. Section~\ref{sec:four-step-block} proves rational four-step closure and gives a direct weighted integral realization. Section~\ref{sec:one-block} gives the local monic-integrality criterion, and Section~\ref{sec:iterated-gauge} classifies complete coefficient orbits by degree product. Section~\ref{sec:point-dynamics} constructs the universal positive four-phase gauge, proves the parabolic affine realization, derives finite conic-identity certificates and periodic modular sieves, and establishes the complementary hyperbolic growth obstruction. Section~\ref{sec:quadratic-coeff} determines primitive periodic quadratic states. Section~\ref{sec:persistent-saturation} treats relation transport, rank collapse, and saturation along Vieta orbits, with the longer constructions in Appendix~\ref{app:persistence-proofs}.

\section{Finite-dimensional reciprocal dynamics and four-step closure}\label{sec:four-step-block}

For fixed $n,m\ge1$, let $\mathcal P_{n,m}$ denote the space of ordered polynomial pairs $(A,B)$ with $\deg A=n$, $\deg B=m$, and nonzero constant terms.  In coefficient coordinates, $\mathcal P_{n,m}$ is a Zariski-open subset of an affine space of dimension $n+m+2$.  Each normalized reciprocal update below is rational in these coordinates.  Four alternating updates return to the same ordered degree pattern and therefore define a rational return map on $\mathcal P_{n,m}$; after the first block, the dynamics restricts to the monic subspace.  The integral coefficient states form the distinguished arithmetic subset whose invariant orbits we classify.

For a degree-$d$ polynomial $P$ and a nonzero parameter $c$, write
\begin{equation}\label{eq:c-reciprocal}
  \Rcal_cP(T)\defeq T^dP(c/T).
\end{equation}
When $P(0)\ne0$, monic reciprocal normalization is obtained by dividing \eqref{eq:c-reciprocal} by $P(0)$.

\begin{proposition}[universal four-step rational gauge closure]\label{prop:four-step-gauge-closure}
Let $A,B\in\Z[T]$ have positive degrees $n,m$, respectively, and nonzero
constant terms
\[
  a\defeq A(0),\qquad b\defeq B(0).
\]
For coefficient dynamics, extend the definition of $\Rcal_cP$ from
\eqref{eq:c-reciprocal} to $c\in\mathbb Q^{\times}$ and
$P\in\mathbb Q[T]$ by the same formula.  Define successively
\begin{align}
  B_1&\defeq\frac1b\Rcal_aB,
  &q&\defeq B_1(0),\label{eq:gauge-four-B1}\\
  A_1&\defeq\frac1a\Rcal_qA,
  &r&\defeq A_1(0),\label{eq:gauge-four-A1}\\
  B_2&\defeq\frac1q\Rcal_rB_1,
  &t&\defeq B_2(0),\label{eq:gauge-four-B2}\\
  A_2&\defeq\frac1r\Rcal_tA_1,
  &u&\defeq A_2(0).\label{eq:gauge-four-A2}
\end{align}
Then $B_1,A_1,B_2,A_2\in\mathbb Q[T]$ are monic.  Writing
$a_n\defeq\operatorname{lc}(A)$ and $b_m\defeq\operatorname{lc}(B)$,
their successive constant terms are
\begin{equation}\label{eq:gauge-four-constants}
  q=\frac{b_ma^m}{b},\qquad
  r=\frac{a_nq^n}{a},\qquad
  t=\frac{r^m}{q},\qquad
  u=\frac{t^n}{r}.
\end{equation}
Moreover,
\begin{equation}\label{eq:gauge-four-return}
  B_2(T)=\frac{t}{b}\,B\!\left(\frac{a}{r}T\right),
  \qquad
  A_2(T)=\frac{u}{a}\,A\!\left(\frac{q}{t}T\right).
\end{equation}
Thus every coefficient pair with nonzero constant terms returns after four
monic rational reciprocal normalizations to its original pair of rational
gauge classes.
\end{proposition}

\begin{proof}
The leading coefficient of $\Rcal_cP$ is $P(0)$.  Hence each division in
\eqref{eq:gauge-four-B1}--\eqref{eq:gauge-four-A2} makes the resulting
polynomial monic, and taking constant terms gives
\eqref{eq:gauge-four-constants}.

For any degree-$d$ polynomial $P\in\mathbb Q[T]$ with $P(0)\ne0$ and nonzero
$v,w\in\mathbb Q$,
\begin{equation}\label{eq:double-reciprocal-gauge}
  \Rcal_w(\Rcal_vP)(T)=w^dP\!\left(\frac vwT\right).
\end{equation}
Using \eqref{eq:double-reciprocal-gauge} twice gives
\[
  B_2(T)
  =\frac{1}{qb}\Rcal_r(\Rcal_aB)(T)
  =\frac{r^m}{qb}B\!\left(\frac arT\right)
  =\frac tbB\!\left(\frac arT\right),
\]
and similarly
\[
  A_2(T)
  =\frac{1}{ra}\Rcal_t(\Rcal_qA)(T)
  =\frac{t^n}{ra}A\!\left(\frac qtT\right)
  =\frac uaA\!\left(\frac qtT\right).
\]
This proves \eqref{eq:gauge-four-return}.
\end{proof}

\subsection{Weighted integral representatives and homogeneous gauges}
\label{subsec:weighted-two-variable}

The rational return concerns polynomial shapes up to scalar and variable
dilation.  Explicitly, two nonzero polynomials $P,Q\in\Q[T]$ are
\emph{gauge equivalent} if
\begin{equation}\label{eq:weighted-gauge-equivalence}
 Q(T)=uP(vT),\qquad u,v\in\Q^\times.
\end{equation}
Before imposing the monic normalization, we record an integral realization
that retains the scalar content as part of the coefficient state.

A \emph{weighted primitive pair} consists of positive integers
$\lambda,\mu$ and nonzero primitive polynomials $A,B\in\Z[T]$.  A
nonzero integer point $(x,y)$ realizes this pair when
\begin{equation}\label{eq:weighted-two-variable-state}
 y\mid\lambda A(x),\qquad x\mid\mu B(y).
\end{equation}
An overall sign is carried by the primitive polynomial, so positive
weights impose no restriction.  For raw transport we also use
$\Rcal_0B(T)\defeq B(0)T^{\deg B}$, the coefficientwise extension of
\eqref{eq:c-reciprocal} to a zero reciprocal parameter.

\begin{proposition}[weighted integral continuation and four-step gauge return]
\label{prop:weighted-two-variable-gauge}
Let $(x,y)$ realize \eqref{eq:weighted-two-variable-state}, and suppose
the canonical quotient
\[
 y'\defeq\frac{\lambda A(x)}y
\]
is nonzero.  Put $c\defeq\lambda A(0)$, $m\defeq\deg B$, and define
\begin{equation}\label{eq:weighted-two-variable-raw}
 \widetilde B(T)\defeq\mu\Rcal_cB(T).
\end{equation}
Then $x\mid\widetilde B(y')$ and $y'\mid\lambda A(x)$.  If
$\widetilde B\ne0$, the assignments
\[
 g\defeq\cont(\widetilde B)>0,\qquad
 B'\defeq\widetilde B/g,\qquad \mu'\defeq g
\]
therefore give a new weighted primitive pair realized by $(x,y')$.
The symmetric assertion holds for an $x$-jump.  Alternating these
updates requires only exact quotient operations and coefficient gcds,
and can continue for as long as the new coordinates and raw transported
polynomials are nonzero.

If $A,B$ have positive degrees and $A(0)B(0)\ne0$, every raw
reciprocal transform preserves its source degree and has nonzero
constant term.  Consequently, whenever four alternating quotient jumps
have nonzero coordinates, each resulting primitive polynomial is gauge
equivalent to its initial polynomial.
\end{proposition}

\begin{proof}
Write $B(T)=\sum_{i=0}^m b_iT^i$.  The raw transform is the integer
polynomial
\[
 \widetilde B(T)=\mu\sum_{i=0}^m b_ic^iT^{m-i}.
\]
Since $yy'=\lambda A(x)\equiv c\pmod x$, one has
\[
 \widetilde B(y')
 \equiv\mu\sum_{i=0}^m b_i(yy')^i(y')^{m-i}
 =\mu(y')^mB(y)
 \equiv0\pmod x.
\]
The other divisibility follows from $yy'=\lambda A(x)$.  Extracting
the content only rewrites the same polynomial as $\widetilde B=\mu'B'$,
so it preserves the relation without cancelling any scalar from a
congruence.  This proves continuation; its operations do not require
integer factorization.

If $c\ne0$ and $b_0b_m\ne0$, the leading and constant coefficients of
$\widetilde B$ are respectively $\mu b_0$ and $\mu b_mc^m$, both
nonzero.  Thus the endpoint hypothesis persists through every update.
In four alternating jumps, each polynomial undergoes two reciprocal
transforms.  The identity \eqref{eq:double-reciprocal-gauge} shows that
their composition is a scalar times a rational variable dilation.
The weights and primitive normalizations add only nonzero scalar
factors, so the resulting polynomial belongs to its initial gauge class.
\end{proof}

This weighted realization and the monic realization answer different
normalization questions.  The weighted state retains the entire raw
integer polynomial through its weight and primitive factor.  The local
criteria in Section~\ref{sec:one-block} determine when the specified
monic representatives themselves have integer coefficients.

The gauge return also has a homogeneous interpretation.  For a
degree-$d$ polynomial with nonzero constant term define
\[
 P^h(U,V)\defeq V^dP(U/V).
\]
For $c\ne0$,
\begin{equation}\label{eq:weighted-homogeneous-action}
 (\Rcal_cP)^h(U,V)=P^h(cV,U).
\end{equation}
One reciprocal update is an anti-diagonal change of homogeneous
coordinates, and two such updates give the diagonal change
$P^h(U,V)\mapsto P^h(vU,wV)$.  Thus four alternating jumps return to
the original orbit under scalar and variable dilations.  In particular,
for a quadratic
\[
 P(T)=aT^2+pT+c,\qquad ac\ne0,
\]
the quantities
\begin{equation}\label{eq:quadratic-gauge-invariants}
 \frac{p^2}{ac},\qquad \frac{p^2-4ac}{ac}
\end{equation}
are gauge invariants: under $P(T)\mapsto uP(vT)$ the three coefficients
become $uav^2,upv,uc$, and both ratios are unchanged.

\section{Integral realization of one four-step block}\label{sec:one-block}

\begin{definition}[endpoint valuation defects]\label{def:endpoint-defects}
Let
\[
  P(T)=\sum_{i=0}^{d}p_iT^i\in\mathbb Q[T],
  \qquad p_0p_d\ne0.
\]
For a prime $p$, define the \emph{left} and \emph{right endpoint defects}
\begin{align}
  \lambda_p(P)
  &\defeq
  \max_{\substack{1\le i\le d\\ p_i\ne0}}
  \left\lceil\frac{v_p(p_0)-v_p(p_i)}{i}\right\rceil,
  \label{eq:left-endpoint-defect}\\
  \rho_p(P)
  &\defeq
  \max_{\substack{0\le i<d\\ p_i\ne0}}
  \left\lceil\frac{v_p(p_d)-v_p(p_i)}{d-i}\right\rceil.
  \label{eq:right-endpoint-defect}
\end{align}
Both maxima are nonempty because $p_d\ne0$ and $p_0\ne0$, respectively.
\end{definition}

The two defects are the endpoint slope constraints of the $p$-adic
coefficient diagram.  They give an exact one-step test for monic
normalization.

\begin{proposition}[one-step monic integrality]\label{prop:one-step-monic-integrality}
For $P$ as in Definition~\ref{def:endpoint-defects} and
$c\in\mathbb Q^{\times}$, put
\begin{equation}\label{eq:monic-reciprocal-normalization}
  \Mcal_cP\defeq\frac{1}{P(0)}\Rcal_cP.
\end{equation}
For every prime $p$, writing $s\defeq v_p(c)$, one has
\begin{equation}\label{eq:one-step-integrality-criterion}
  \Mcal_cP\in\mathbb Z_p[T]
  \quad\Longleftrightarrow\quad
  s\ge\lambda_p(P).
\end{equation}
Moreover,
\begin{equation}\label{eq:endpoint-defect-transform}
  \lambda_p(\Mcal_cP)=s+\rho_p(P),
  \qquad
  \rho_p(\Mcal_cP)=\lambda_p(P)-s.
\end{equation}
\end{proposition}

\begin{proof}
The coefficient of $T^{d-i}$ in \eqref{eq:monic-reciprocal-normalization}
has $p$-adic valuation
\[
  v_p(p_i)+is-v_p(p_0).
\]
These valuations are all nonnegative precisely when
\[
  s\ge
  \left\lceil\frac{v_p(p_0)-v_p(p_i)}{i}\right\rceil
\]
for every nonzero coefficient $p_i$ with $1\le i\le d$, which is
\eqref{eq:one-step-integrality-criterion}.

Write $Q\defeq\Mcal_cP$.  The coefficient of $T^j$ in $Q$ is
$p_{d-j}c^{d-j}/p_0$.  Substitution into
\eqref{eq:left-endpoint-defect} gives
\[
  \lambda_p(Q)
  =s+\max_{0\le i<d,\ p_i\ne0}
  \left\lceil\frac{v_p(p_d)-v_p(p_i)}{d-i}\right\rceil
  =s+\rho_p(P).
\]
The same calculation from the other endpoint gives
$\rho_p(Q)=\lambda_p(P)-s$.
\end{proof}

\begin{proposition}[exact integral realization of the four-step cycle]\label{prop:four-step-integrality}
Retain the notation of Proposition~\ref{prop:four-step-gauge-closure}.  For
a prime $p$, write
\[
  \alpha\defeq v_p(a),\qquad
  \beta\defeq v_p(b),\qquad
  \gamma\defeq v_p(q),\qquad
  \delta\defeq v_p(r),\qquad
  \varepsilon\defeq v_p(t).
\]
Equivalently, with $a_n\defeq\operatorname{lc}(A)$ and
$b_m\defeq\operatorname{lc}(B)$,
\begin{equation}\label{eq:four-step-valuation-recursion}
  \gamma=v_p(b_m)+m\alpha-\beta,
  \qquad
  \delta=v_p(a_n)+n\gamma-\alpha,
  \qquad
  \varepsilon=m\delta-\gamma.
\end{equation}
Then
\[
  B_1,A_1,B_2,A_2\in\mathbb Z[T]
\]
if and only if, for every prime $p$,
\begin{equation}\label{eq:four-step-integrality-criterion}
\begin{aligned}
  \alpha&\ge\lambda_p(B),&
  \gamma&\ge\lambda_p(A),\\
  \delta-\alpha&\ge\rho_p(B),&
  \varepsilon-\gamma&\ge\rho_p(A).
\end{aligned}
\end{equation}
Thus the universal rational four-step gauge cycle has an integral monic
realization exactly when the four endpoint-slope inequalities
\eqref{eq:four-step-integrality-criterion} hold locally at every prime.
\end{proposition}

\begin{proof}
The first two states are
\[
  B_1=\Mcal_aB,
  \qquad
  A_1=\Mcal_qA.
\]
Hence Proposition~\ref{prop:one-step-monic-integrality} gives the first
two inequalities in \eqref{eq:four-step-integrality-criterion}.  For the
third state,
\[
  B_2=\Mcal_rB_1,
\]
and \eqref{eq:endpoint-defect-transform} gives
\[
  \lambda_p(B_1)=\alpha+\rho_p(B).
\]
Thus $B_2\in\mathbb Z_p[T]$ exactly when
$\delta-\alpha\ge\rho_p(B)$.  Similarly,
\[
  \lambda_p(A_1)=\gamma+\rho_p(A),
\]
so $A_2=\Mcal_tA_1$ is $p$-integral exactly when
$\varepsilon-\gamma\ge\rho_p(A)$.  A rational polynomial lies in
$\mathbb Z[T]$ precisely when it lies in $\mathbb Z_p[T]$ for every
prime $p$, which proves the result.
\end{proof}

\begin{corollary}[finite obstruction set for the four-step cycle]\label{cor:four-step-obstruction-primes}
If $nm\ge2$, then every prime $p\nmid A(0)B(0)$ automatically satisfies
\eqref{eq:four-step-integrality-criterion}.  Consequently, for every pair
of positive degrees other than the linear--linear case, integral monic
realization of the universal four-step cycle is decided by the finitely
many primes dividing $A(0)B(0)$.
\end{corollary}

\begin{proof}
Fix $p\nmid ab$, so $\alpha=\beta=0$.  Since $A,B\in\mathbb Z[T]$,
\[
  \lambda_p(A)\le0,
  \qquad
  \lambda_p(B)\le0.
\]
Put
\[
  \eta\defeq v_p(a_n),
  \qquad
  \theta\defeq v_p(b_m).
\]
Then \eqref{eq:four-step-valuation-recursion} gives
\[
  \gamma=\theta,
  \qquad
  \delta=\eta+n\theta,
  \qquad
  \varepsilon-\gamma
  =m\eta+(mn-2)\theta.
\]
Also
\[
  \rho_p(B)\le\theta,
  \qquad
  \rho_p(A)\le\eta.
\]
Hence the third inequality follows from
$\delta\ge\theta$, and the fourth from
$m\eta+(mn-2)\theta\ge\eta$ when $mn\ge2$.
\end{proof}

\begin{remark}\label{rem:linear-linear-extra-primes}
The exclusion of the linear--linear case in
Corollary~\ref{cor:four-step-obstruction-primes} is necessary.  For
\[
  A(T)=T+1,
  \qquad
  B(T)=2T+1,
\]
one has $A(0)B(0)=1$, but the successive monic states are
\[
  B_1(T)=T+2,
  \quad
  A_1(T)=T+2,
  \quad
  B_2(T)=T+1,
  \quad
  A_2(T)=T+\frac12.
\]
Thus the new obstruction prime $2$ can enter through a leading
coefficient when $n=m=1$.
\end{remark}

\begin{proposition}[denominator propagation]\label{prop:four-step-denominators}
Write
\[
  A(T)=\sum_{i=0}^{n}a_iT^i,
  \qquad
  B(T)=\sum_{j=0}^{m}b_jT^j,
\]
and retain $\alpha,\beta,\gamma,\delta,\varepsilon$ from
Proposition~\ref{prop:four-step-integrality}.  For a rational polynomial
$Q$, let $d_p(Q)$ denote the exponent of $p$ in the least common
denominator of its coefficients.  With $v_p(0)=+\infty$, one has
\begin{align}
  d_p(B_1)
  &=\max\left(0,
    \max_{1\le j\le m}
    \{\beta-v_p(b_j)-j\alpha\}\right),
  \label{eq:denom-B1}\\
  d_p(A_1)
  &=\max\left(0,
    \max_{1\le i\le n}
    \{\alpha-v_p(a_i)-i\gamma\}\right),
  \label{eq:denom-A1}\\
  d_p(B_2)
  &=\max\left(0,
    \max_{0\le j<m}
    \{v_p(b_m)-v_p(b_j)-(m-j)(\delta-\alpha)\}\right),
  \label{eq:denom-B2}\\
  d_p(A_2)
  &=\max\left(0,
    \max_{0\le i<n}
    \{v_p(a_n)-v_p(a_i)-(n-i)(\varepsilon-\gamma)\}\right).
  \label{eq:denom-A2}
\end{align}
In particular, Proposition~\ref{prop:four-step-integrality} is equivalent
to the vanishing of these four denominator exponents at every prime.
\end{proposition}

\begin{proof}
For the first two states, the coefficient of $T^{m-j}$ in $B_1$ is
$b_ja^j/b$, while the coefficient of $T^{n-i}$ in $A_1$ is
$a_iq^i/a$; taking the negative parts of their valuations gives
\eqref{eq:denom-B1}--\eqref{eq:denom-A1}.

For the last two states, use the gauge-return formulas
\eqref{eq:gauge-four-return}.  The coefficient of $T^j$ in $B_2$ has
valuation
\[
  v_p(b_j)-v_p(b_m)+(m-j)(\delta-\alpha),
\]
where \eqref{eq:four-step-valuation-recursion} has been used to eliminate
the scalar factor.  This gives \eqref{eq:denom-B2}.  The same calculation
for $A_2$ gives \eqref{eq:denom-A2}.
\end{proof}

\begin{remark}[monic denominators versus primitive content]\label{rem:monic-denominator-content}
The denominator formulas also admit a direct primitive-content interpretation.  If $P\in\mathbb Z[T]$,
$c\in\mathbb Z\setminus\{0\}$, and
\[
  g\defeq\cont(\Rcal_cP),
\]
then $g\mid |P(0)|$.  If $D$ is the least positive integer for which
$D\Mcal_cP\in\mathbb Z[T]$, then
\begin{equation}\label{eq:monic-denominator-content}
  D=\frac{|P(0)|}{g},
  \qquad
  D\Mcal_cP
  =\operatorname{sgn}(P(0))\frac{\Rcal_cP}{g}.
\end{equation}
Thus the least denominator of the monic normalization measures exactly
how much of the endpoint scalar remains after removing the full primitive
content of the raw reciprocal transform.
\end{remark}

\begin{corollary}[common-constant indefinite integral realization]\label{cor:common-constant-integral-gauge}
Suppose $A,B\in\Z[T]$ have degrees $n,m\ge2$ and
\[
  A(0)=B(0)=c\ne0.
\]
Then every polynomial produced by the forward alternating canonical monic
reciprocal normalizations is monic and lies in $\Z[T]$.  In particular,
the four polynomials $B_1,A_1,B_2,A_2$ of
Proposition~\ref{prop:four-step-gauge-closure} are integral and primitive.
If $A$ and $B$ are monic with constant term $1$, then
$c=q=r=t=1$, so the first four steps give a literal return.
\end{corollary}

\begin{proof}
We use the following divisibility invariant.  Suppose the reciprocal
parameter is $a\in\Z\setminus\{0\}$ and the polynomial to be
transported is
\[
 Q(T)=b+\sum_{i=1}^d q_iT^i\in\Z[T],
 \qquad d\ge2,\qquad b\mid a.
\]
Then
\[
 \Mcal_aQ(T)=T^d+
 \sum_{i=1}^d q_i\frac ab\,a^{i-1}T^{d-i}
 \in\Z[T].
\]
Its constant term $b'\defeq q_da^d/b$ satisfies
\[
 \frac{b'}a=q_d\frac ab\,a^{d-2}\in\Z.
\]
Thus, after swapping the polynomial roles, the constant term of the next
polynomial to be transported again divides the next reciprocal
parameter.  Initially this holds because the two constants equal $c$.
Induction proves that every newly produced monic polynomial is integral;
being monic, it is also primitive.  If the initial pair is monic with
constant term $1$, all four constants in
\eqref{eq:gauge-four-constants} equal $1$, and the return formulas
\eqref{eq:gauge-four-return} give the asserted literal return.
\end{proof}

\begin{remark}[coefficient closure versus arithmetic realization]\label{rem:rational-gauge-versus-integral}
Proposition~\ref{prop:four-step-gauge-closure} is purely coefficient-level,
whereas Proposition~\ref{prop:four-step-integrality} is exact at the
integer-coefficient level.  The common-constant hypothesis of
Corollary~\ref{cor:common-constant-integral-gauge} is sufficient but not
necessary.  For example,
\[
  A(T)=T^2-3T+2,
  \qquad
  B(T)=T^2-3T+1
\]
has unequal constant terms but the four monic states are
\[
  B_1=T^2-6T+4,
  \quad
  A_1=T^2-6T+8,
  \quad
  B_2=T^2-12T+16,
  \quad
  A_2=T^2-12T+32.
\]
Late obstructions can also occur.  For
\[
  A(T)=T^2+2T+4,
  \qquad
  B(T)=T^2+2T+8,
\]
the first two monic states are $T^2+T+2$ and $T^2+T+1$, but
\[
  B_2(T)=T^2+\frac12T+\frac12.
\]
Likewise, for
\[
  A(T)=T^2+2T+8,
  \qquad
  B(T)=T^2+4T+16,
\]
the first three monic states are integral while
\[
  A_2(T)=T^2+\frac12T+\frac12.
\]
Thus all four inequalities in
\eqref{eq:four-step-integrality-criterion} are genuinely needed.  Integral coefficient realization is a property of the polynomial states themselves; whether a particular integer point follows those states is a separate arithmetic question.
\end{remark}

\section{Iterated integral dynamics by degree product}\label{sec:iterated-gauge}

The one-block criterion above can be iterated because after one universal
four-step block both coefficient polynomials are monic.  For a monic pair
$A,B\in\Z[T]$ of positive degrees $n,m$, with nonzero constant terms
\[
  a\defeq A(0),\qquad b\defeq B(0),\qquad s\defeq nm,
\]
write $\Phi(A,B)=(A^+,B^+)$ for the final pair after one complete
four-step block.  For a prime $p$, put
\[
  \alpha\defeq v_p(a),\qquad \beta\defeq v_p(b).
\]
The constant-term formulas \eqref{eq:gauge-four-constants} give
\begin{equation}\label{eq:block-M}
  \begin{pmatrix}
    v_p(A^+(0))\\[1mm]
    v_p(B^+(0))
  \end{pmatrix}
  =M_{n,m}
  \begin{pmatrix}\alpha\\ \beta\end{pmatrix},
  \qquad
  M_{n,m}\defeq
  \begin{pmatrix}
    s^2-3s+1&-n(s-2)\\
    m(s-2)&-(s-1)
  \end{pmatrix}.
\end{equation}
One has $\det M_{n,m}=1$ and
$\operatorname{tr}M_{n,m}=s^2-4s+2$.  In the standard $\operatorname{SL}_2$
classification, $s=1,3$ are elliptic of finite order, $s=2$ gives the
central finite-order matrix $-I$, $s=4$ is parabolic, and $s>4$ is
hyperbolic.  We group the cases $s<4$ as the finite-order, elliptic-type
regime.

\begin{proposition}[finite-order low-product blocks]\label{prop:elliptic-blocks}
Suppose $s=nm<4$.  The forward monic coefficient orbit of $(A,B)$ is
integral indefinitely if and only if
\begin{equation}\label{eq:elliptic-unit-constants}
  |A(0)|=|B(0)|=1.
\end{equation}
Under \eqref{eq:elliptic-unit-constants}, the complete-block dynamics is
periodic.  If $s\in\{1,3\}$, then
\begin{equation}\label{eq:elliptic-block-map-13}
  \Phi(A,B)
  =\bigl(bA(bT),\,abB(abT)\bigr),
  \qquad
  \Phi^3(A,B)=(A,B).
\end{equation}
If $(n,m)=(1,2)$, then
\begin{equation}\label{eq:elliptic-block-map-12}
  \Phi(A,B)=\bigl(A(T),\,B(bT)\bigr),
  \qquad
  \Phi^2(A,B)=(A,B),
\end{equation}
and if $(n,m)=(2,1)$, then
\begin{equation}\label{eq:elliptic-block-map-21}
  \Phi(A,B)=\bigl(A(aT),\,B(T)\bigr),
  \qquad
  \Phi^2(A,B)=(A,B).
\end{equation}
Consequently forward and two-sided integral realization are equivalent in
every low-product degree pattern, and both are characterized by
\eqref{eq:elliptic-unit-constants}.
\end{proposition}

\begin{proof}
For $s=2$, the matrix in \eqref{eq:block-M} is $-I$.  If the forward
orbit is integral, then $(\alpha,\beta)$ and
$(-\alpha,-\beta)$ are both componentwise nonnegative, hence
$\alpha=\beta=0$.  For $s=1$ or $s=3$, the same matrix satisfies
\[
  M_{n,m}^2+M_{n,m}+I=0.
\]
The three valuation vectors
$v,M_{n,m}v,M_{n,m}^2v$, where $v=(\alpha,\beta)^T$, are componentwise
nonnegative along an integral forward orbit and sum to zero.  Hence again
$\alpha=\beta=0$.  Since this holds for every prime,
$|a|=|b|=1$.

Conversely, signed-unit constant terms make every reciprocal parameter and
every monic normalization divisor a signed unit, so every intermediate
polynomial remains integral.  Substitution into
\eqref{eq:gauge-four-constants}--\eqref{eq:gauge-four-return} gives
\eqref{eq:elliptic-block-map-13} for $s=1,3$ and
\eqref{eq:elliptic-block-map-12}--\eqref{eq:elliptic-block-map-21} for
$s=2$.  The displayed periods follow by direct iteration.  In particular
the same finite integral cycle extends in both block directions.
\end{proof}

The parabolic case has a different rigidity: the block does not usually
return periodically, but it acts by a fixed rational dilation.

\begin{proposition}[parabolic degree-product blocks]\label{prop:parabolic-blocks}
Suppose $s=nm=4$.  Define
\begin{equation}\label{eq:parabolic-CD}
  C\defeq\frac{a^2}{b^n},
  \qquad
  D\defeq\frac{a^m}{b^2}.
\end{equation}
Then $C^m=D^2$ and $D^n=C^2$, and one complete block is
\begin{equation}\label{eq:parabolic-block-map}
  \Phi(A,B)
  =\left(D^nA(T/D),\,C^mB(T/C)\right).
\end{equation}
The two dilation parameters $C,D$ are invariant under $\Phi$.  Hence, if
$(A_k,B_k)\defeq\Phi^k(A,B)$, then for every $k\ge0$,
\begin{equation}\label{eq:parabolic-block-iterates}
  A_k(T)=D^{nk}A(T/D^k),
  \qquad
  B_k(T)=C^{mk}B(T/C^k).
\end{equation}

Write $B_{k,1},A_{k,1}$ for the first two intermediate monic states in
block $k$, and let $B_{0,1},A_{0,1}$ be those in the initial block.  Then
\begin{equation}\label{eq:parabolic-intermediate-iterates}
  B_{k,1}(T)=C^{mk}B_{0,1}(T/C^k),
  \qquad
  A_{k,1}(T)=D^{nk}A_{0,1}(T/D^k).
\end{equation}
Consequently the forward monic coefficient orbit is integral indefinitely
if and only if the first complete block is integral and
\begin{equation}\label{eq:parabolic-forward-criterion}
  C=\frac{A(0)^2}{B(0)^n}\in\Z.
\end{equation}
It is integral in both block directions if and only if the first complete
block is integral and
\begin{equation}\label{eq:parabolic-two-sided-criterion}
  C\in\{\pm1\}.
\end{equation}
In the latter case the complete-block coefficient orbit has period at most
$2$.
\end{proposition}

\begin{proof}
For a monic pair with $s=4$, the constant terms in
\eqref{eq:gauge-four-constants} are
\[
  q=\frac{a^m}{b},\qquad
  r=\frac{a^3}{b^n},\qquad
  t=\frac{a^{2m}}{b^3},\qquad
  u=\frac{a^5}{b^{2n}}.
\]
Substitution into \eqref{eq:gauge-four-return} gives
\eqref{eq:parabolic-block-map}.  Its final constant terms are
$u=aD^n$ and $t=bC^m$, and a direct calculation gives
\[
  \frac{u^2}{t^n}=C,
  \qquad
  \frac{u^m}{t^2}=D,
\]
so $C,D$ are block invariants.  Induction proves
\eqref{eq:parabolic-block-iterates}.  Applying the first two normalizations
to the scaled pair gives \eqref{eq:parabolic-intermediate-iterates}.

For a prime $p$, put
\begin{equation}\label{eq:parabolic-dp}
  d_p\defeq v_p(C)=2\alpha-n\beta.
\end{equation}
Then
\[
  v_p(D)=\frac m2d_p
\]
and \eqref{eq:block-M} becomes
\begin{equation}\label{eq:parabolic-valuation-drift}
  v_p(A_k(0))=\alpha+2kd_p,
  \qquad
  v_p(B_k(0))=\beta+mk d_p.
\end{equation}
If the forward orbit is integral, these quantities cannot remain
nonnegative when $d_p<0$, so $d_p\ge0$ for every $p$, which is equivalent
to $C\in\Z$.  Conversely, $C\in\Z$ implies $D\in\Z$ because
$v_p(D)=m d_p/2\ge0$ for every prime.  If the first block is integral,
then \eqref{eq:parabolic-block-iterates} and
\eqref{eq:parabolic-intermediate-iterates} show that every later state is
integral.

For a two-sided integral orbit, \eqref{eq:parabolic-valuation-drift} must
remain nonnegative for all $k\in\Z$, forcing $d_p=0$ for every prime.
Thus $C\in\{\pm1\}$, and then also $D\in\{\pm1\}$.  Conversely, under
these signed-unit dilations the formulas above preserve integrality for
all $k\in\Z$ and give period at most $2$.
\end{proof}

\begin{corollary}[quadratic specialization]\label{cor:quadratic-block-iteration}
Let
\[
  A(T)=T^2+a_1T+a,
  \qquad
  B(T)=T^2+b_1T+b,
  \qquad ab\ne0,
\]
be monic quadratics and put $c\defeq a/b$.  Then
\begin{equation}\label{eq:quadratic-block-iterates}
  A_k(T)=T^2+a_1c^{2k}T+ac^{4k},
  \qquad
  B_k(T)=T^2+b_1c^{2k}T+bc^{4k}.
\end{equation}
The first two intermediate states in block $k$ are
\begin{align}
  B_{k,1}(T)&=T^2+b_1c^{2k+1}T+ac^{4k+1},\label{eq:quadratic-block-B1}\\
  A_{k,1}(T)&=T^2+a_1c^{2k+1}T+ac^{4k+2}.\label{eq:quadratic-block-A1}
\end{align}
The forward orbit is integral indefinitely if and only if $c\in\Z$, and
it is integral in both block directions if and only if $c=\pm1$.

For a prime $p$, put $e_p\defeq v_p(c)$ and use $v_p(0)=+\infty$.
If $e_p<0$, then all states through $N$ complete blocks are $p$-integral
if and only if
\begin{equation}\label{eq:quadratic-N-block-criterion}
\begin{aligned}
  v_p(a)+4Ne_p&\ge0,&
  v_p(b)+4Ne_p&\ge0,\\
  v_p(a_1)+2Ne_p&\ge0,&
  v_p(b_1)+2Ne_p&\ge0.
\end{aligned}
\end{equation}
If $e_p\ge0$, there is no $p$-adic obstruction.
\end{corollary}

\begin{proof}
Here $C=D=c^2$, so \eqref{eq:parabolic-block-iterates} and
\eqref{eq:parabolic-intermediate-iterates} give
\eqref{eq:quadratic-block-iterates}--\eqref{eq:quadratic-block-A1}.
Moreover $c^2\in\Z$ for $c\in\mathbb Q$ if and only if $c\in\Z$.  Thus
$c\in\Z$ makes the first block integral and Proposition~\ref{prop:parabolic-blocks}
gives the forward and two-sided assertions.  If $e_p<0$, every relevant
coefficient valuation in \eqref{eq:quadratic-block-iterates}--
\eqref{eq:quadratic-block-A1} decreases with $k$; evaluating the last
states through block $N$ gives exactly
\eqref{eq:quadratic-N-block-criterion}.  If $e_p\ge0$, all these
valuations are nondecreasing.
\end{proof}

\begin{corollary}[arbitrary low-product initial pair]\label{cor:low-product-general}
Let $A,B\in\Z[T]$ have positive degrees $n,m$ with $nm\le4$ and nonzero
constant terms, and suppose their first complete universal four-step block
is integral.  Let $(A_*,B_*)$ be the resulting monic pair, with constants
$a_*\defeq A_*(0)$ and $b_*\defeq B_*(0)$.

If $nm<4$, the forward coefficient orbit is integral indefinitely if and only if
$|a_*|=|b_*|=1$; in that case the resulting monic pair lies on a two-sided
integral periodic orbit.  If $nm=4$, the forward orbit is integral indefinitely if and only if
\begin{equation}\label{eq:parabolic-postblock-general}
  \frac{a_*^2}{b_*^n}\in\Z,
\end{equation}
and the resulting monic pair has a two-sided integral continuation if and only if this ratio lies in $\{\pm1\}$.
Writing $a_n\defeq\operatorname{lc}(A)$ and
$b_m\defeq\operatorname{lc}(B)$, the parabolic ratio in
\eqref{eq:parabolic-postblock-general} is
\begin{equation}\label{eq:parabolic-original-endpoints}
  \frac{a_*^2}{b_*^n}
  =a_n^2b_m^n\frac{A(0)^2}{B(0)^n}.
\end{equation}
In particular, for arbitrary quadratic initial pairs the forward condition
reduces to
\[
  \operatorname{lc}(A)\operatorname{lc}(B)
  \frac{A(0)}{B(0)}\in\Z,
\]
in agreement with the quadratic specialization above.
\end{corollary}

\begin{proof}
After the first complete block the pair is monic, so
Propositions~\ref{prop:elliptic-blocks} and~\ref{prop:parabolic-blocks}
apply.  Formula \eqref{eq:parabolic-original-endpoints} follows by
substituting \eqref{eq:gauge-four-constants}; when $n=m=2$ its right-hand
side is the square of the displayed quadratic ratio, and a rational number
whose square is an integer is itself an integer.
\end{proof}

\subsection{Hyperbolic coefficient orbits}\label{sec:hyperbolic-coefficient-orbits}

The low-product coefficient cases are now classified.  We treat every
hyperbolic pattern $nm>4$ at once.  Let
$A,B\in\Z[T]$ be monic of positive degrees $n,m$, respectively, with
nonzero constant terms.  Fix a prime $p$, put
\[
  s\defeq nm>4,
  \qquad
  \alpha\defeq v_p(A(0)),
  \qquad
  \beta\defeq v_p(B(0)),
\]
and define $(u_j)_{j\ge0}$ by
\begin{equation}\label{eq:hyperbolic-u-recurrence}
  u_0\defeq\beta,
  \qquad
  u_1\defeq\alpha,
  \qquad
  u_{j+1}\defeq
  \begin{cases}
    m u_j-u_{j-1},&j\text{ odd},\\
    n u_j-u_{j-1},&j\text{ even}.
  \end{cases}
\end{equation}
Thus
\begin{align}
  u_2&=m\alpha-\beta,\notag\\
  u_3&=(s-1)\alpha-n\beta,\notag\\
  u_4&=m(s-2)\alpha-(s-1)\beta,\notag\\
  u_5&=(s^2-3s+1)\alpha-n(s-2)\beta.
  \label{eq:hyperbolic-first-valuations}
\end{align}

\begin{proposition}[hyperbolic degree-product reduction]\label{prop:hyperbolic-blocks}
Let $(A_k,B_k)$ denote the monic coefficient pair after $k$ complete
four-step blocks.  Then
\begin{equation}\label{eq:hyperbolic-block-constants}
  v_p(B_k(0))=u_{4k},
  \qquad
  v_p(A_k(0))=u_{4k+1}.
\end{equation}
Equivalently,
\begin{equation}\label{eq:hyperbolic-block-matrix}
  \begin{pmatrix}
    v_p(A_{k+1}(0))\\[1mm]
    v_p(B_{k+1}(0))
  \end{pmatrix}
  =M_{n,m}
  \begin{pmatrix}
    v_p(A_k(0))\\[1mm]
    v_p(B_k(0))
  \end{pmatrix},
\end{equation}
where $M_{n,m}$ is the matrix in \eqref{eq:block-M}; equivalently,
\begin{equation}\label{eq:hyperbolic-M}
  M_{n,m}
  =\begin{pmatrix}
    s^2-3s+1&-n(s-2)\\
    m(s-2)&-(s-1)
  \end{pmatrix}.
\end{equation}
Moreover
\begin{equation}\label{eq:hyperbolic-half-matrix}
  M_{n,m}=H_{n,m}^2,
  \qquad
  H_{n,m}\defeq
  \begin{pmatrix}
    s-1&-n\\
    m&-1
  \end{pmatrix},
\end{equation}
with
\[
  \det H_{n,m}=\det M_{n,m}=1,
  \qquad
  \operatorname{tr}M_{n,m}=s^2-4s+2.
\]
Writing
\begin{equation}\label{eq:hyperbolic-half-eigenvalues}
  \xi_\pm\defeq
  \frac{s-2\pm\sqrt{s(s-4)}}2,
\end{equation}
the eigenvalues of $H_{n,m}$ are $\xi_\pm$, while those of $M_{n,m}$
are $\xi_\pm^2$.

Define the four local endpoint quantities
\begin{align}
  L_B&\defeq m\lambda_p(B)-\beta,
  &L_A&\defeq n\lambda_p(A)-\alpha,
  \label{eq:hyperbolic-left-invariants}\\
  R_B&\defeq m\rho_p(B)+\beta,
  &R_A&\defeq n\rho_p(A)+\alpha.
  \label{eq:hyperbolic-right-invariants}
\end{align}
They are invariant under a complete four-step block.  The entire forward
monic coefficient orbit is $p$-integral if and only if, for every $k\ge0$,
\begin{equation}\label{eq:hyperbolic-all-block-criterion}
\begin{aligned}
  u_{4k+2}&\ge L_B,&
  u_{4k+3}&\ge L_A,\\
  u_{4k+4}&\ge R_B,&
  u_{4k+5}&\ge R_A.
\end{aligned}
\end{equation}
\end{proposition}

\begin{proof}
For a monic pair, the constant valuations in one block are exactly
$u_0,u_1,\ldots,u_5$ from
\eqref{eq:hyperbolic-u-recurrence}.  This gives
\eqref{eq:hyperbolic-block-constants}--\eqref{eq:hyperbolic-M}, and direct
multiplication gives \eqref{eq:hyperbolic-half-matrix}.  The determinant,
trace, and eigenvalue statements follow immediately.

Write $B'$ and $A'$ for the two polynomials after one complete block.
The gauge-return formula \eqref{eq:gauge-four-return}, or equivalently two
applications of \eqref{eq:endpoint-defect-transform}, gives
\[
  \lambda_p(B')=u_3+\lambda_p(B)-u_1,
  \qquad
  \rho_p(B')=u_1+\rho_p(B)-u_3,
\]
and
\[
  \lambda_p(A')=u_4+\lambda_p(A)-u_2,
  \qquad
  \rho_p(A')=u_2+\rho_p(A)-u_4.
\]
Using
$u_2=mu_1-u_0$, $u_3=nu_2-u_1$,
$u_4=mu_3-u_2$, and $u_5=nu_4-u_3$
shows that \eqref{eq:hyperbolic-left-invariants}--
\eqref{eq:hyperbolic-right-invariants} are unchanged.
At block $k$, the four inequalities of
Proposition~\ref{prop:four-step-integrality} are respectively equivalent
to the four inequalities in \eqref{eq:hyperbolic-all-block-criterion}.
\end{proof}

\begin{corollary}[irrational forward cone and finite valley test]\label{cor:hyperbolic-cone}
Put
\begin{equation}\label{eq:hyperbolic-kappa}
  \kappa_\pm\defeq
  \frac{s\pm\sqrt{s(s-4)}}{2n}
  =\frac m2\left(1\pm\sqrt{1-\frac4s}\right).
\end{equation}
For nonnegative integers $\alpha,\beta$, all constant valuations $u_j$ in
\eqref{eq:hyperbolic-u-recurrence} are nonnegative if and only if
\begin{equation}\label{eq:hyperbolic-forward-cone}
  (\alpha,\beta)=(0,0)
  \quad\text{or}\quad
  \alpha>0,\qquad 0\le\beta<\kappa_+\alpha.
\end{equation}
Thus the closure of the nonzero forward constant-valuation cone is
\[
  0\le\beta\le
  \frac{s+\sqrt{s(s-4)}}{2n}\,\alpha.
\]
Its upper boundary is the contracting eigenline of $H_{n,m}$ and is
irrational for every integer $s>4$; the expanding eigenline
$\beta=\kappa_-\alpha$ lies in the interior.

Put
\begin{equation}\label{eq:hyperbolic-tau}
  \tau\defeq s^2-4s+2>2.
\end{equation}
For every $r\in\{2,3,4,5\}$, the subsequence
\[
  x_k^{(r)}\defeq u_{4k+r}
\]
satisfies
\begin{equation}\label{eq:hyperbolic-subsequence-recurrence}
  x_{k+2}^{(r)}=\tau x_{k+1}^{(r)}-x_k^{(r)}.
\end{equation}
For a nonzero lattice point in \eqref{eq:hyperbolic-forward-cone}, each
such subsequence tends to $+\infty$ and has a finite valley: its minimum
is attained at the first index $k$ for which
$x_{k+1}^{(r)}\ge x_k^{(r)}$, and possibly also at $k+1$ when equality
holds.  Hence \eqref{eq:hyperbolic-all-block-criterion} is decided by
checking the four subsequences only through their respective valleys.
If $(\alpha,\beta)$ lies outside the forward cone, then some residue
subsequence tends to $-\infty$, so indefinite $p$-integrality is
impossible.
\end{corollary}

\begin{proof}
The two eigenvectors of $H_{n,m}$ may be chosen as
$(1,\kappa_-)^T$ and $(1,\kappa_+)^T$, corresponding to
$\xi_+$ and $\xi_-$, respectively.  For
$v=(\alpha,\beta)^T$, the coefficient of the expanding eigenvector is
\[
  \frac{\kappa_+\alpha-\beta}{\kappa_+-\kappa_-}.
\]
Thus a vector above the upper eigenline eventually acquires negative
coordinates under positive powers of $H_{n,m}$.  If
$0\le\beta<\kappa_+\alpha$, then
$u_2=m\alpha-\beta>0$ and
$u_3=(s-1)\alpha-n\beta>0$.  Moreover the map
\[
  f(r)\defeq\frac{m-r}{s-1-nr}
\]
has positive denominator on $0\le r\le\kappa_+$, since
$s-1-n\kappa_+=\xi_->0$, and
\[
  f'(r)=\frac{1}{(s-1-nr)^2}>0.
\]
Because $f(\kappa_+)=\kappa_+$, it follows that
$f([0,\kappa_+))\subset[0,\kappa_+)$.  Hence $H_{n,m}$ preserves the cone and induction gives
$u_j\ge0$ for all $j$.  The boundary cannot contain a nonzero lattice
point: if $s(s-4)=d^2$ with integers $s>4,d\ge0$, then
\[
  (s-2-d)(s-2+d)=4,
\]
and the only positive factorization with equal parity gives $s=4$.

Every four-step subsequence satisfies the characteristic equation of
$M_{n,m}$, which is \eqref{eq:hyperbolic-subsequence-recurrence}.  Inside
the nonzero forward cone its expanding component is positive and the
contracting boundary contains no nonzero lattice point, so the subsequence
tends to $+\infty$.  Once
$x_{k+1}^{(r)}\ge x_k^{(r)}\ge0$,
\[
  x_{k+2}^{(r)}-x_{k+1}^{(r)}
  =(\tau-1)x_{k+1}^{(r)}-x_k^{(r)}>0.
\]
This proves the valley assertion.  Outside the cone, the negative
expanding component forces a residue subsequence to $-\infty$.
\end{proof}

\begin{corollary}[finite decision procedure in every hyperbolic degree pattern]\label{cor:hyperbolic-decision}
Let $A,B\in\Z[T]$ have positive degrees $n,m$ with $nm>4$ and nonzero
constant terms.  Their universal monic coefficient orbit is integral
indefinitely if and only if the first complete four-step block is integral
and, for the resulting monic pair, each prime dividing its two constant
terms satisfies the forward-cone condition
\eqref{eq:hyperbolic-forward-cone} and the four valley minima satisfy
\eqref{eq:hyperbolic-all-block-criterion}.  Thus indefinite coefficient
integrality is decidable by finitely many local calculations.
\end{corollary}

\begin{proof}
After the first complete block both polynomials are monic.  Apply
Proposition~\ref{prop:hyperbolic-blocks} and
Corollary~\ref{cor:hyperbolic-cone}.  A prime not dividing either constant
term of the resulting monic pair has $\alpha=\beta=0$ and nonpositive
endpoint defects, hence is automatically unobstructed.
\end{proof}

\begin{proposition}[two-sided hyperbolic integral realization]\label{prop:hyperbolic-bilateral}
Retain the monic hyperbolic setting of
Proposition~\ref{prop:hyperbolic-blocks}.  Extend
\eqref{eq:hyperbolic-u-recurrence} uniquely to all $j\in\mathbb Z$ by the
same reversible recurrence, and let $\kappa_\pm$ be as in
\eqref{eq:hyperbolic-kappa}.  Then all constant valuations $u_j$ are
nonnegative for $j\in\mathbb Z$ if and only if
\begin{equation}\label{eq:hyperbolic-bilateral-cone}
  (\alpha,\beta)=(0,0)
  \quad\text{or}\quad
  \alpha>0,\qquad
  \kappa_-\alpha<\beta<\kappa_+\alpha.
\end{equation}

The complete-block map on monic rational coefficient pairs is reversible.
Its entire bi-infinite orbit, including all intermediate monic states
inside every block, is $p$-integral if and only if
\eqref{eq:hyperbolic-bilateral-cone} holds and
\begin{equation}\label{eq:hyperbolic-bilateral-criterion}
\begin{aligned}
  \min_{k\in\mathbb Z}u_{4k+2}&\ge L_B,&
  \min_{k\in\mathbb Z}u_{4k+3}&\ge L_A,\\
  \min_{k\in\mathbb Z}u_{4k+4}&\ge R_B,&
  \min_{k\in\mathbb Z}u_{4k+5}&\ge R_A,
\end{aligned}
\end{equation}
where the four invariant endpoint quantities are those of
\eqref{eq:hyperbolic-left-invariants}--
\eqref{eq:hyperbolic-right-invariants}.

For every $r\in\{2,3,4,5\}$, put $x_k^{(r)}\defeq u_{4k+r}$ and
\begin{equation}\label{eq:hyperbolic-Lambda}
  \Lambda\defeq\xi_+^2
  =\left(\frac{s-2+\sqrt{s(s-4)}}2\right)^2>1.
\end{equation}
In the nonzero cone \eqref{eq:hyperbolic-bilateral-cone}, one has
\begin{equation}\label{eq:hyperbolic-bilateral-closed-form}
  x_k^{(r)}=C_r\Lambda^k+D_r\Lambda^{-k},
  \qquad C_r,D_r>0,
\end{equation}
so $x_k^{(r)}\to+\infty$ as $k\to\pm\infty$ and each minimum in
\eqref{eq:hyperbolic-bilateral-criterion} is attained at a finite global
valley.  If $x_0^{(r)},x_1^{(r)}$ are known, then
\begin{equation}\label{eq:hyperbolic-bilateral-CD}
  C_r=\frac{x_1^{(r)}-\Lambda^{-1}x_0^{(r)}}
             {\Lambda-\Lambda^{-1}},
  \qquad
  D_r=\frac{\Lambda x_0^{(r)}-x_1^{(r)}}
             {\Lambda-\Lambda^{-1}}.
\end{equation}
The discrete minimum is attained at one or both integers neighboring
\begin{equation}\label{eq:hyperbolic-bilateral-continuous-min}
  t_r\defeq\frac{\log(D_r/C_r)}{2\log\Lambda}.
\end{equation}
Thus the two-sided criterion is finite and effective.
\end{proposition}

\begin{proof}
Let
\[
  e_-\defeq(1,\kappa_-)^T,
  \qquad
  e_+\defeq(1,\kappa_+)^T.
\]
By Corollary~\ref{cor:hyperbolic-cone}, these are respectively the
expanding and contracting eigenvectors of $H_{n,m}$, with positive
eigenvalues $\xi_+$ and $\xi_-$.  Writing
\[
  (\alpha,\beta)^T=C e_-+D e_+,
\]
gives
\[
  C=\frac{\kappa_+\alpha-\beta}{\kappa_+-\kappa_-},
  \qquad
  D=\frac{\beta-\kappa_-\alpha}{\kappa_+-\kappa_-}.
\]
If all forward and backward half-block iterates are nonnegative, the
limits in the expanding direction as the exponent tends to $+\infty$ and
$-\infty$ force $C,D\ge0$.  Conversely, $C,D\ge0$ makes every
$H_{n,m}^j(\alpha,\beta)^T$ componentwise nonnegative for
$j\in\mathbb Z$.  Neither eigenline contains a nonzero lattice point,
because the two slopes are irrational for $s>4$, as proved in
Corollary~\ref{cor:hyperbolic-cone}.  This proves
\eqref{eq:hyperbolic-bilateral-cone}.

A monic reciprocal normalization is reversible: if $Q$ is monic of degree
$d$, $a\ne0$, and
\[
  Q'\defeq\Mcal_aQ,
\]
then one has
\[
  Q=\Mcal_aQ'.
\]
Indeed $Q'(0)=a^d/Q(0)$ because $Q$ is monic, while
$\Rcal_aQ'=a^dQ/Q(0)$ by $\Rcal_a^2Q=a^dQ$.  Hence reversing the four individual
normalizations gives the inverse complete block over $\mathbb Q$.

The endpoint quantities $L_B,L_A,R_B,R_A$ are block invariants by
Proposition~\ref{prop:hyperbolic-blocks}; the same algebraic identities
hold for negative block indices.  Block invariance gives
\[
  \rho_p(B_k)=\frac{R_B-u_{4k}}m,
  \qquad
  \rho_p(A_k)=\frac{R_A-u_{4k+1}}n.
\]
For a monic rational polynomial, $\rho_p\le0$ is equivalent to
$p$-integrality of all its coefficients.  The $R_B$ and $R_A$ inequalities in
\eqref{eq:hyperbolic-bilateral-criterion}, applied one block earlier,
therefore make every starting pair $(A_k,B_k)$ $p$-integral.  Once that is
known, Proposition~\ref{prop:four-step-integrality} applies at block $k$:
its four local conditions are precisely the four inequalities in
\eqref{eq:hyperbolic-all-block-criterion}, with the $L_B,L_A$ inequalities
controlling the two intermediate states and the $R_B,R_A$ inequalities
controlling the next pair.  Conversely, if the entire bi-infinite orbit is
$p$-integral, the same proposition at every block gives all four
inequalities.  This proves the equivalence with
\eqref{eq:hyperbolic-bilateral-criterion} without assuming backward
integrality in advance.

Finally, each residue sequence satisfies
\eqref{eq:hyperbolic-subsequence-recurrence}, whose roots are
$\Lambda$ and $\Lambda^{-1}$.  In the strict bilateral cone, the two
eigencomponents are positive; after any fixed phase shift the same is
true, which gives \eqref{eq:hyperbolic-bilateral-closed-form} with
$C_r,D_r>0$.  Formula \eqref{eq:hyperbolic-bilateral-CD} follows from the
values at $k=0,1$.  The real-variable function
$C_r\Lambda^t+D_r\Lambda^{-t}$ is strictly convex and has its unique
minimum at \eqref{eq:hyperbolic-bilateral-continuous-min}; hence its
restriction to $\mathbb Z$ is minimized at one or both neighboring
integers.  The zero vector is immediate.  This proves the finite valley
test.
\end{proof}

\begin{corollary}[finite two-sided decision procedure in every hyperbolic degree pattern]\label{cor:hyperbolic-bilateral-decision}
Let $A,B\in\mathbb Z[T]$ be monic of positive degrees $n,m$ with
$nm>4$ and nonzero constant terms.  Their rational monic reciprocal block
orbit is integral for every block index $k\in\mathbb Z$ if and only if,
for every prime dividing $A(0)B(0)$, the local data satisfy
Proposition~\ref{prop:hyperbolic-bilateral}.  Primes not dividing
$A(0)B(0)$ are automatically unobstructed.  Thus two-sided integral
coefficient realization is decidable by finitely many local calculations.
\end{corollary}

\begin{proof}
If $p\nmid A(0)B(0)$, then $\alpha=\beta=0$.  Since $A,B$ are monic and
integral, all four endpoint defects are nonpositive, so the zero sequence
satisfies \eqref{eq:hyperbolic-bilateral-criterion}.  The result follows
from Proposition~\ref{prop:hyperbolic-bilateral} at the finitely many
remaining primes.
\end{proof}

\begin{example}[a nonunit two-sided hyperbolic orbit]\label{ex:hyperbolic-bilateral-nonunit}
Let $p$ be a prime and take
\[
  A(T)\defeq T^2+p^2,
  \qquad
  B(T)\defeq T^3+p^3.
\]
At the prime $p$, one has $(\alpha,\beta)=(2,3)$ and
\[
  \kappa_-\alpha
  =(3-\sqrt3)
  <3<
  (3+\sqrt3)
  =\kappa_+\alpha.
\]
Moreover
\[
  L_B=L_A=R_B=R_A=0.
\]
Thus Proposition~\ref{prop:hyperbolic-bilateral} gives an integral monic
coefficient orbit in both block directions.  In particular, two-sided
hyperbolic integrality is not restricted to signed-unit constant terms.
\end{example}

\begin{example}[the quadratic--cubic hyperbolic pattern]\label{ex:23-hyperbolic}
For $(n,m)=(2,3)$ one has $s=6$ and
\[
  M_{2,3}=
  \begin{pmatrix}
    19&-8\\
    12&-5
  \end{pmatrix},
  \qquad
  \kappa_+=\frac{3+\sqrt3}{2},
  \qquad
  \tau=14.
\]
Thus Proposition~\ref{prop:hyperbolic-blocks} specializes to the four
recurrence inequalities of the first mixed hyperbolic case.  For example,
\[
  A(T)=T^2+2^{11},
  \qquad
  B(T)=T^3+2^{16}T+2^{26}
\]
has one complete integral block ending at
\[
  T^2+2,
  \qquad
  T^3+T+4,
\]
but the next first normalization is
\[
  \Mcal_2(T^3+T+4)=T^3+\frac12T^2+2.
\]
\end{example}

The valley bound is pair-dependent in every hyperbolic degree pattern.
In fact the obstruction can be postponed arbitrarily far.

\begin{proposition}[arbitrarily delayed hyperbolic obstruction]\label{prop:hyperbolic-delayed}
Fix positive degrees $n,m$ with $s\defeq nm>4$ and let $M_{n,m}$ be
\eqref{eq:hyperbolic-M}.  For every $N\ge0$ there exists a monic integral
pair of degrees $(n,m)$ that admits exactly $N+1$ complete integral
four-step blocks and then fails in the following block.  More explicitly:

\smallskip
\noindent\emph{(a) Suppose $m\ge2$.}
Put
\[
  v_*\defeq
  \begin{pmatrix}n\\ s-1\end{pmatrix},
  \qquad
  \begin{pmatrix}\alpha_N\\ \beta_N\end{pmatrix}
  \defeq M_{n,m}^{-(N+1)}v_*,
\]
and
\begin{equation}\label{eq:hyperbolic-delayed-e}
  e_N\defeq\frac{(m-1)\beta_N-(m+1)}m.
\end{equation}
Then $\alpha_N,\beta_N,e_N$ are nonnegative integers and
\begin{equation}\label{eq:hyperbolic-delayed-family-mge2}
  A_N(T)\defeq T^n+2^{\alpha_N},
  \qquad
  B_N(T)\defeq T^m+2^{e_N}T+2^{\beta_N}
\end{equation}
has the required lifetime.  After $N+1$ blocks it reaches
\begin{equation}\label{eq:hyperbolic-terminal-mge2}
  A_*(T)=T^n+2^n,
  \qquad
  B_*(T)=T^m+2^{(m-1)n-2}T+2^{s-1},
\end{equation}
and the next first normalization is
\begin{equation}\label{eq:hyperbolic-terminal-failure-mge2}
  \Mcal_{2^n}B_*(T)=T^m+\frac12T^{m-1}+2.
\end{equation}

\smallskip
\noindent\emph{(b) Suppose $m=1$; then $n>4$.}
Put
\[
  v_*\defeq
  \begin{pmatrix}n-1\\ n-2\end{pmatrix},
  \qquad
  \begin{pmatrix}\alpha_N\\ \beta_N\end{pmatrix}
  \defeq M_{n,1}^{-(N+1)}v_*,
\]
and
\begin{equation}\label{eq:hyperbolic-delayed-e-linear}
  e_N\defeq\frac{(n-1)\alpha_N-(n+1)}n.
\end{equation}
Then $\alpha_N,\beta_N,e_N$ are nonnegative integers and
\begin{equation}\label{eq:hyperbolic-delayed-family-linear}
  A_N(T)\defeq T^n+2^{e_N}T+2^{\alpha_N},
  \qquad
  B_N(T)\defeq T+2^{\beta_N}
\end{equation}
has the required lifetime.  After $N+1$ blocks it reaches
\begin{equation}\label{eq:hyperbolic-terminal-linear}
  A_*(T)=T^n+2^{n-3}T+2^{n-1},
  \qquad
  B_*(T)=T+2^{n-2}.
\end{equation}
The first normalization in the following block is integral, but the second
is
\begin{equation}\label{eq:hyperbolic-terminal-failure-linear}
  \Mcal_2A_*(T)=T^n+\frac12T^{n-1}+2,
\end{equation}
so the block fails.
\end{proposition}

\begin{proof}
Because $M_{n,m}\in\operatorname{SL}_2(\Z)$, all vectors appearing in the
statement are integral.

Suppose first that $m\ge2$.  Reduction of $M_{n,m}^{-1}$ modulo $n$ and
$m$ shows inductively that
\[
  \alpha_N\equiv0\pmod n,
  \qquad
  \beta_N\equiv-1\pmod m,
\]
so \eqref{eq:hyperbolic-delayed-e} is integral.  The first backward
iterate of $v_*$ is
\begin{equation}\label{eq:hyperbolic-delayed-first-backward}
  \alpha_0=n(s-1)(s-3),
  \qquad
  \beta_0=s^3-5s^2+6s-1,
\end{equation}
and both entries are positive for $s\ge5$; in fact $\beta_0\ge29$,
so $e_0\ge0$.  The slope of $v_*$ is
$(s-1)/n>\kappa_+$.  In the eigenbasis from
Corollary~\ref{cor:hyperbolic-cone}, $v_*$ therefore has a positive
contracting-eigenline component and a negative expanding-eigenline
component.  Under successive powers of $M_{n,m}^{-1}$ the former grows
and the latter decreases in absolute value.  Since the first backward
iterate is positive, all later $\alpha_N,\beta_N$ are positive and
increase with $N$; hence $e_N\ge e_0\ge0$.  The same argument applied
after multiplication by $H_{n,m}$ shows that the two intermediate constant
valuations also increase with $N$.

For the family \eqref{eq:hyperbolic-delayed-family-mge2}, a direct
reciprocal calculation gives
\[
  A'(T)=T^n+2^{u_5},
  \qquad
  B'(T)=T^m+2^{e'}T+2^{u_4},
\]
where
\[
  e'=e+(m-1)\bigl((s-2)\alpha-n\beta\bigr).
\]
If $e=((m-1)\beta-(m+1))/m$, then
$e'=((m-1)u_4-(m+1))/m$.  Hence one block sends the family with valuation
vector $v_N$ to the same family with vector $v_{N-1}$, and sends $v_0$ to
$v_*$.  Its invariant endpoint thresholds are
\begin{equation}\label{eq:hyperbolic-delayed-thresholds-mge2}
  L_B=m+1,
  \qquad L_A=R_A=0,
  \qquad
  R_B=
  \begin{cases}
    3,&m=2,\\
    m-1,&m\ge3.
  \end{cases}
\end{equation}
For $v_0$, the four left sides of
\eqref{eq:hyperbolic-all-block-criterion} are
\[
  s^2-3s+1,
  \qquad n(s-2),
  \qquad s-1,
  \qquad n,
\]
which satisfy \eqref{eq:hyperbolic-delayed-thresholds-mge2}; the spectral
monotonicity just noted shows that the inequalities remain valid for every
$v_N$.  Thus each of the first $N+1$ blocks is integral.  Formula
\eqref{eq:hyperbolic-terminal-mge2} is exactly the family associated with
$v_*$, while \eqref{eq:hyperbolic-terminal-failure-mge2} is immediate.

Now suppose $m=1$.  Reduction of $M_{n,1}^{-1}$ modulo $n$ gives
$\alpha_N\equiv-1\pmod n$, so
\eqref{eq:hyperbolic-delayed-e-linear} is integral.  The first backward
iterate is
\[
  \alpha_0=n^3-5n^2+6n-1,
  \qquad
  \beta_0=(n-2)(n^2-4n+2),
\]
and is positive for $n\ge5$.  Moreover
$e_0=n^3-6n^2+11n-8>0$, and the later $e_N$ increase with $\alpha_N$.
Also $(n-2)/(n-1)>\kappa_+$; indeed, after squaring the equivalent positive
inequality, the difference is $4n$.  The same eigenbasis argument therefore
shows positivity of all later backward iterates and their half-block
intermediates.

For \eqref{eq:hyperbolic-delayed-family-linear}, one block preserves the
family because
\[
  e'=e+(n-1)\bigl((n-3)\alpha-(n-2)\beta\bigr)
\]
and the relation
$e=((n-1)\alpha-(n+1))/n$ is carried to the same relation for the new
constant valuation.  The invariant thresholds are
\begin{equation}\label{eq:hyperbolic-delayed-thresholds-linear}
  L_B=R_B=0,
  \qquad L_A=n+1,
  \qquad R_A=n-1.
\end{equation}
For the first backward iterate the four left sides of
\eqref{eq:hyperbolic-all-block-criterion} are
\[
  (n-3)(n-1),
  \qquad n^2-3n+1,
  \qquad n-2,
  \qquad n-1,
\]
and satisfy \eqref{eq:hyperbolic-delayed-thresholds-linear}; again all
earlier blocks satisfy them by spectral monotonicity.  The terminal formulas
\eqref{eq:hyperbolic-terminal-linear}--
\eqref{eq:hyperbolic-terminal-failure-linear} then finish the proof.
\end{proof}

\begin{remark}\label{rem:hyperbolic-no-uniform-horizon}
Proposition~\ref{prop:hyperbolic-delayed} shows that the finite valley test
of Corollary~\ref{cor:hyperbolic-cone} cannot be replaced by any fixed
number of initial blocks depending only on $(n,m)$.  The local obstruction
prime may remain the same and the endpoint thresholds may stay invariant
while the first failed inequality moves arbitrarily far down the same
hyperbolic recurrence.
\end{remark}

\section{Point dynamics after coefficient gauges}\label{sec:point-dynamics}

The coefficient-integrality criteria do not by themselves determine the
growth of the compatible quotient coordinates.  This section constructs
periodic coefficient shapes, gives an affine realization at the parabolic
degree boundary, and compares it with hyperbolic growth of normalized
point orbits.  Finite identity certificates and modular sieves record two
arithmetic consequences of the affine realization.

\subsection{Positive four-phase coefficient gauges}
\label{sec:positive-coefficient-gauges}

We first describe the effect of a periodic coefficient gauge on a compatible
point orbit, and then construct such a gauge for every positive-constant
canonical monic chain.

Unfold an alternating monic coefficient chain into a single sequence
$(P_j)$ so that a compatible point orbit $(x_j)$ satisfies
\begin{equation}\label{eq:hyperbolic-point-chain}
  x_{j-1}x_{j+1}=P_j(x_j).
\end{equation}
The canonical monic reciprocal update is then
\begin{equation}\label{eq:hyperbolic-P-recurrence}
  P_{j+1}
  =\frac{1}{P_{j-1}(0)}\,
    \Rcal_{P_j(0)}P_{j-1}.
\end{equation}
Its degrees alternate between $n$ and $m$.

\begin{lemma}[periodic-shape coefficient gauges]
\label{lem:periodic-shape-gauge}
Assume that $P_j(0)>0$ and that there are a positive integer $q$, fixed
monic polynomials $Q_0,\ldots,Q_{q-1}\in\mathbb R[T]$ with positive
constant terms, and positive numbers $\delta_j$ such that
\begin{equation}\label{eq:periodic-shape-gauge}
  P_j(T)=\delta_j^{d_j}
  Q_{j\bmod q}(T/\delta_j),
  \qquad d_j\defeq\deg P_j.
\end{equation}
Then
\begin{equation}\label{eq:periodic-shape-lambda}
  \lambda_j\defeq
  \frac{\delta_j^{d_j}}{\delta_{j-1}\delta_{j+1}}
\end{equation}
depends only on $j\bmod q$.  If
\begin{equation}\label{eq:periodic-shape-point-gauge}
  u_j\defeq \frac{x_j}{\delta_j},
\end{equation}
then the point orbit satisfies the fixed periodic polynomial recurrence
\begin{equation}\label{eq:periodic-shape-point-recurrence}
  u_{j-1}u_{j+1}
  =\lambda_{j\bmod q}\,Q_{j\bmod q}(u_j).
\end{equation}
Moreover any two positive gauges of the form
\eqref{eq:periodic-shape-gauge}, after passing to a common period, differ
only by a positive periodic multiplier.  Hence their normalized point
sequences differ only by a positive periodic multiplier as well.
\end{lemma}

\begin{proof}
Write $r=j\bmod q$ and $q_r\defeq Q_r(0)$.  Substituting
\eqref{eq:periodic-shape-gauge} into
\eqref{eq:hyperbolic-P-recurrence}, and then replacing $T$ by
$\delta_{j+1}T$, gives
\[
  Q_{r+1}(T)
  =\frac{1}{q_{r-1}}
    \Rcal_{\theta_j}Q_{r-1}(T),
  \qquad
  \theta_j=q_r\lambda_j.
\]
Taking constant terms yields
\[
  q_{r+1}=\frac{\theta_j^{d_{j-1}}}{q_{r-1}}.
\]
All quantities are positive, so $\theta_j$, and therefore $\lambda_j$,
is determined by the phase $r$ alone.  Substituting
$x_j=\delta_ju_j$ into \eqref{eq:hyperbolic-point-chain} now gives
\eqref{eq:periodic-shape-point-recurrence}.

For the last assertion, suppose also
\[
  P_j(T)=\widehat\delta_j^{d_j}
  \widehat Q_{j\bmod q}(T/\widehat\delta_j)
\]
after replacing the two periods by a common multiple.  Comparing constant
terms gives
\[
  \left(\frac{\widehat\delta_j}{\delta_j}\right)^{d_j}
  =\frac{Q_{j\bmod q}(0)}{\widehat Q_{j\bmod q}(0)}.
\]
Positivity makes the ratio $\widehat\delta_j/\delta_j$ depend only on the
phase.  The same is therefore true of the two normalized point sequences.
\end{proof}

\begin{proposition}[universal four-phase coefficient gauge]
\label{prop:universal-four-phase-gauge}
Let $(P_j)_{j\in\mathbb Z}$ be a canonical monic reciprocal coefficient
chain over $\mathbb Q$, with alternating positive degrees $d_j$ and
positive constant terms $c_j\defeq P_j(0)$.  Thus
\[
 P_{j+1}=\frac{1}{c_{j-1}}\Rcal_{c_j}P_{j-1},
 \qquad d_{j+2}=d_j.
\]
Put $Q_r\defeq P_r$ for $0\le r<4$.  There is a positive rational
four-phase gauge
\[
 P_j(T)=\delta_j^{d_j}Q_{j\bmod4}(T/\delta_j)
\]
with $\delta_0=\delta_1=\delta_2=\delta_3=1$.  It is given uniquely by
the reversible recursion
\begin{equation}\label{eq:universal-four-phase-delta}
 \delta_{j+4}=\frac{c_{j+3}}{c_{j+1}}\delta_j.
\end{equation}
The periodic multipliers of Lemma~\ref{lem:periodic-shape-gauge} are
\[
 \lambda_0=\frac{c_2}{c_0},\qquad
 \lambda_1=\lambda_2=1,\qquad
 \lambda_3=\frac{c_1}{c_3}.
\]
The same conclusion holds for a one-sided canonical monic chain with
positive initial constant terms: its coefficient chain has a unique
positive-constant rational continuation to negative indices, whether or
not its point orbit has any corresponding integral continuation.
\end{proposition}

\begin{proof}
Taking constant terms gives
\[
 c_{j+2}c_j=c_{j+1}^{d_j}.
\]
Two reciprocal transformations of $P_j$ consequently give
\begin{align*}
 P_{j+4}(T)
 &=\frac{1}{c_{j+2}c_j}\Rcal_{c_{j+3}}
       (\Rcal_{c_{j+1}}P_j)(T)\\
 &=\frac{c_{j+3}^{d_j}}{c_{j+2}c_j}
       P_j\!\left(\frac{c_{j+1}}{c_{j+3}}T\right)\\
 &=\left(\frac{c_{j+3}}{c_{j+1}}\right)^{d_j}
       P_j\!\left(\frac{c_{j+1}}{c_{j+3}}T\right).
\end{align*}
Induction in both directions along each residue class modulo $4$ proves
the asserted gauge.  All its factors are positive rational numbers.
Formula~\eqref{eq:universal-four-phase-delta} at $j=-1$ and $j=0$
gives $\delta_{-1}=c_0/c_2$ and $\delta_4=c_3/c_1$.
Substitution in
$\lambda_j=\delta_j^{d_j}/(\delta_{j-1}\delta_{j+1})$ gives the four
listed multipliers; their periodicity follows from
Lemma~\ref{lem:periodic-shape-gauge}.

For the final assertion, monic reciprocal normalization is an involution
when its parameter is fixed.  Thus the defining recurrence recovers
$P_{j-1}$ uniquely from $P_j,P_{j+1}$.  The backward constant term is
$c_{j-1}=c_j^{d_{j+1}}/c_{j+1}>0$, so this continuation preserves
positive rational constant terms.  This concerns the coefficient chain
alone and places no additional condition on a point orbit.
\end{proof}

\subsection{C-finite point dynamics at the parabolic boundary}\label{sec:parabolic-cfinite}

The coefficient classification above has a pointwise counterpart.  We call a
sequence \emph{C-finite} if it satisfies a homogeneous linear recurrence with
constant coefficients.  The next elementary lemma records the mechanism that
will occur in all three degree patterns with $nm=4$.

\begin{lemma}[two-periodic affine recurrences]\label{lem:two-periodic-cfinite}
Let $(w_j)$ be a sequence over a field of characteristic zero satisfying
\begin{equation}\label{eq:two-periodic-affine}
  w_{j+1}=\alpha_jw_j-w_{j-1}-\beta_j,
  \qquad
  \alpha_{j+2}=\alpha_j,
  \qquad
  \beta_{j+2}=\beta_j.
\end{equation}
Put
\begin{equation}\label{eq:two-periodic-tau}
  \tau\defeq\alpha_0\alpha_1-2.
\end{equation}
Then $(w_j)$ is C-finite and, wherever the displayed terms are defined,
\begin{equation}\label{eq:two-periodic-order-six}
  w_{j+6}-(\tau+1)w_{j+4}+(\tau+1)w_{j+2}-w_j=0.
\end{equation}
\end{lemma}

\begin{proof}
Set
\[
  s_j\defeq
  \begin{pmatrix}w_j\\ w_{j-1}\\ 1\end{pmatrix},
  \qquad
  L_j\defeq
  \begin{pmatrix}
    \alpha_j&-1&-\beta_j\\
    1&0&0\\
    0&0&1
  \end{pmatrix}.
\]
Then $s_{j+1}=L_js_j$ and $L_{j+2}=L_j$.  For either parity the two-step
matrix has an upper-left $2\times2$ block of determinant $1$ and trace
$\tau$, while the last diagonal entry is $1$.  Hence its characteristic
polynomial is
\[
  (X-1)(X^2-\tau X+1)
  =X^3-(\tau+1)X^2+(\tau+1)X-1.
\]
Cayley--Hamilton on each parity subsequence gives
\eqref{eq:two-periodic-order-six}.
\end{proof}

\begin{theorem}[C-finite realization at the parabolic degree boundary]
\label{thm:parabolic-cfinite}
Let $A,B\in\Q[T]$ be monic polynomials with nonzero constant terms and
positive degrees $n,m$ satisfying $nm=4$.  Suppose a nondegenerate rational point
orbit realizes the successive canonical monic reciprocal coefficient
states.  After grouping two quotient moves into one half-state, the point
coordinates admit a rational geometric gauge for which the normalized
dynamics is a two-periodic affine recurrence of the form
\eqref{eq:two-periodic-affine}.  Consequently every coordinate sequence at
a fixed phase of the jump cycle is C-finite.  The original, ungauged point
coordinates are also C-finite.

More explicitly, the three degree patterns have the following forms.

\begin{enumerate}[label=\textup{(\roman*)}]
\item If $(n,m)=(2,2)$ and
\[
  A(T)=T^2+a_1T+a,
  \qquad
  B(T)=T^2+b_1T+b,
  \qquad
  c\defeq\frac ab,
\]
then after the geometric gauge described in the proof the normalized
points lie on one fixed conic
\begin{equation}\label{eq:parabolic-22-conic}
  U^2+cV^2+a_1U+cb_1V+a=HUV.
\end{equation}
If $w_{2j}=V_j$ and $w_{2j+1}=U_j$, then
\begin{align}
  w_{2j+2}&=\frac Hc\,w_{2j+1}-w_{2j}-b_1,
  \label{eq:parabolic-22-rec-even}\\
  w_{2j+3}&=H w_{2j+2}-w_{2j+1}-a_1.
  \label{eq:parabolic-22-rec-odd}
\end{align}
Thus \eqref{eq:two-periodic-order-six} holds with
\begin{equation}\label{eq:parabolic-22-tau}
  \tau=\frac{H^2}{c}-2.
\end{equation}

\item If $(n,m)=(1,4)$ and
\begin{equation}\label{eq:parabolic-14-polynomials}
  A(T)=T+a,
  \qquad
  B(T)=T^4+pT^3+qT^2+rT+b,
  \qquad
  C\defeq\frac{a^2}{b},
\end{equation}
then there is an invariant $K\in\Q$ for which a normalized scalar
coordinate $(u_j)$ satisfies
\begin{equation}\label{eq:parabolic-14-affine-recurrence}
  u_{j+1}=
  \begin{cases}
    CK u_j-u_{j-1}-Cr/a,&j\text{ even},\\[1mm]
    K u_j-u_{j-1}-p,&j\text{ odd}.
  \end{cases}
\end{equation}
Hence \eqref{eq:two-periodic-order-six} holds with
\begin{equation}\label{eq:parabolic-14-tau}
  \tau=CK^2-2.
\end{equation}

\item The case $(n,m)=(4,1)$ is obtained from \textup{(ii)} by
interchanging the two polynomial and point coordinates.
\end{enumerate}
\end{theorem}

\begin{proof}
We treat the three degree patterns separately.

Suppose first that $(n,m)=(2,2)$.  Keep the notation of
Corollary~\ref{cor:quadratic-block-iteration}.  At the $j$th half-state,
write its point as $(X_j,Y_j)$ and put
\begin{equation}\label{eq:parabolic-22-point-gauge}
  X_j=c^jU_j,
  \qquad
  Y_j=c^jV_j.
\end{equation}
The formulas
\eqref{eq:quadratic-block-iterates}--\eqref{eq:quadratic-block-A1}
show that this removes all coefficient dilation.  The two quotient
relations become
\begin{align}
  V_jV_{j+1}
  &=\frac1c\bigl(U_j^2+a_1U_j+a\bigr),
  \label{eq:parabolic-22-fixed-A}\\
  U_jU_{j+1}
  &=c\bigl(V_{j+1}^2+b_1V_{j+1}+b\bigr).
  \label{eq:parabolic-22-fixed-B}
\end{align}
Because $a=cb$, the two replacements are precisely the two Vieta
reflections on \eqref{eq:parabolic-22-conic}.  Thus the value $H$ is
preserved.  Taking the other root first in $V$ and then in $U$ gives
\eqref{eq:parabolic-22-rec-even}--\eqref{eq:parabolic-22-rec-odd}.
Lemma~\ref{lem:two-periodic-cfinite} applies, with
$\alpha_0\alpha_1=H^2/c$.  The sequences $(U_j)$ and $(V_j)$ are
subsequences of $(w_j)$, hence C-finite; multiplication by the geometric
factor $c^j$ in \eqref{eq:parabolic-22-point-gauge} preserves
C-finiteness.

Now suppose $(n,m)=(1,4)$ and use the notation
\eqref{eq:parabolic-14-polynomials}.  Put
\begin{equation}\label{eq:parabolic-14-Bhat}
  \widehat B(T)
  \defeq\frac1b\Rcal_aB(T)
  =T^4+\frac{rC}{a}T^3+qCT^2+paCT+a^2C.
\end{equation}
Proposition~\ref{prop:parabolic-blocks} and
\eqref{eq:parabolic-intermediate-iterates} show that the successive
half-state coefficient pairs are
\begin{equation}\label{eq:parabolic-14-half-states}
  A^{(j)}(T)=T+aC^j,
\end{equation}
with
\begin{align}
  B^{(2k)}(T)&=C^{4k}B(T/C^k),
  \label{eq:parabolic-14-Beven}\\
  B^{(2k+1)}(T)&=C^{4k}\widehat B(T/C^k).
  \label{eq:parabolic-14-Bodd}
\end{align}
Let $(X_j,Y_j)$ be the point at the $j$th half-state.  Compatibility
with the two quotient moves means
\begin{equation}\label{eq:parabolic-14-point-products}
  Y_jY_{j+1}=A^{(j)}(X_j),
  \qquad
  X_{j-1}X_j=B^{(j)}(Y_j).
\end{equation}
Define
\begin{equation}\label{eq:parabolic-14-point-gauge}
  Y_{2k}=C^ku_{2k},
  \qquad
  Y_{2k+1}=C^ku_{2k+1},
  \qquad
  X_j=C^jz_j.
\end{equation}
Then the first relation in \eqref{eq:parabolic-14-point-products} gives
\begin{equation}\label{eq:parabolic-14-z}
  z_j=u_ju_{j+1}-a,
\end{equation}
and the second gives
\begin{equation}\label{eq:parabolic-14-Qproducts}
  z_{j-1}z_j=
  \begin{cases}
    Q_0(u_j),&j\text{ even},\\
    Q_1(u_j),&j\text{ odd},
  \end{cases}
\end{equation}
where
\begin{align}
  Q_0(T)&\defeq CT^4+CpT^3+CqT^2+CrT+a^2,
  \label{eq:parabolic-14-Q0}\\
  Q_1(T)&\defeq C^{-1}T^4+\frac raT^3+qT^2+paT+a^2.
  \label{eq:parabolic-14-Q1}
\end{align}
The products in \eqref{eq:parabolic-14-Qproducts} are equivalent to
alternating Vieta reflections on the two conics
\begin{align}
  C^{-1}u^2+\frac rau+v^2+pv+q
  &=K(uv-a),
  \label{eq:parabolic-14-conic0}\\
  u^2+pu+C^{-1}v^2+\frac rav+q
  &=K(uv-a).
  \label{eq:parabolic-14-conic1}
\end{align}
Indeed, regarding \eqref{eq:parabolic-14-conic0} as a quadratic in $u$,
the other root is
\[
  CKv-u-\frac{Cr}{a},
\]
while regarding \eqref{eq:parabolic-14-conic1} as a quadratic in $u$,
the other root is
\[
  Kv-u-p.
\]
Direct substitution shows that each replacement changes
\eqref{eq:parabolic-14-conic0} into
\eqref{eq:parabolic-14-conic1}, or conversely, with the same $K$; its
root product is exactly \eqref{eq:parabolic-14-Qproducts}.  This proves
\eqref{eq:parabolic-14-affine-recurrence}.  Applying
Lemma~\ref{lem:two-periodic-cfinite} gives
\eqref{eq:two-periodic-order-six} with
$\tau=CK^2-2$.

The actual $Y$-coordinates are also C-finite.  More explicitly,
\begin{equation}\label{eq:parabolic-14-Y-recurrence}
  Y_{j+6}
  -C(CK^2-1)Y_{j+4}
  +C^2(CK^2-1)Y_{j+2}
  -C^3Y_j=0.
\end{equation}
Finally \eqref{eq:parabolic-14-z} gives
$X_j=C^j(u_ju_{j+1}-a)$.  Products of C-finite sequences are C-finite
(for instance by tensoring companion matrices), and multiplication by a
geometric sequence preserves C-finiteness, so $(X_j)$ is C-finite as
well.

The case $(4,1)$ follows by interchanging the two coordinate roles and
the two polynomial roles.  This completes the proof.
\end{proof}

The affine realization also gives a finite verification bound for algebraic
identities on a normalized orbit.  Khomovsky~\cite{Khomovsky2018} develops a
related approach to identities among Lucas sequences through linear systems
for the unknown coefficients and interpolation.  The following is a conic
specialization of the finite-dimensional principle underlying C-finite identity
algorithms~\cite{KauersZimmermann2008,Zeilberger2013}.

\begin{proposition}[finite identity certificates on an invariant conic]
\label{prop:conic-identity-certificate}
Let $\mathbb F$ be a field of characteristic zero, let
$Q\in\mathbb F[U,V]$ have total degree $2$, and let
$T\colon\mathbb F^2\to\mathbb F^2$ be an affine map satisfying
$Q\circ T=\lambda Q$ for some $\lambda\in\mathbb F$.
Suppose $s_0\in\mathbb F^2$ satisfies $Q(s_0)=0$, and put
$s_j\defeq T^j(s_0)$ for $j\ge0$.
For an integer $h\ge0$ and $P\in\mathbb F[U,V]$ of total degree at most $h$,
\[
 P(s_j)=0\quad\text{for all }j\ge0
 \quad\Longleftrightarrow\quad
 P(s_j)=0\quad(0\le j\le2h).
\]
Consequently, for any finite list $P_1,\ldots,P_t$ of polynomials of degree
at most $h$, all coefficient vectors $c\in\mathbb F^t$ for which
$\sum_{i=1}^t c_iP_i(s_j)=0$ for every $j\ge0$ form exactly the kernel of
the $(2h+1)\times t$ evaluation matrix
$\bigl(P_i(s_j)\bigr)_{0\le j\le2h,\,1\le i\le t}$.
\end{proposition}

\begin{proof}
Write $\mathcal P_h\defeq\mathbb F[U,V]_{\le h}$, with
$\mathcal P_e\defeq\{0\}$ when $e<0$.
Because $Q$ has degree $2$ and the polynomial ring is a domain,
\[
 (Q)\cap\mathcal P_h=Q\mathcal P_{h-2},
 \qquad
 \dim\bigl(\mathcal P_h/Q\mathcal P_{h-2}\bigr)=2h+1.
\]
For $h\ge2$, the dimension follows by subtracting
$\binom{h}{2}$ from $\binom{h+2}{2}$; for $h=0,1$ it is immediate.
Affine substitution preserves degree, and $Q\circ T=\lambda Q$ makes it
well-defined on this quotient.  Evaluation at $s_0$ also descends to the
quotient because $Q(s_0)=0$.
By Cayley--Hamilton, the sequence $P(T^j(s_0))$ is annihilated by the
monic characteristic polynomial of this $(2h+1)$-dimensional substitution
operator.  Its first $2h+1$ values therefore determine all later values.
This proves the equivalence; applying it to $\sum_i c_iP_i$ proves the
kernel assertion.  No smoothness or irreducibility of the conic is used.
\end{proof}

For the normalized $(2,2)$ orbit of
Theorem~\ref{thm:parabolic-cfinite}, one half-state step is the composition
of the two affine Vieta reflections on
\eqref{eq:parabolic-22-conic}; hence the proposition applies directly.
For $(1,4)$, two scalar steps preserve each of the conics
\eqref{eq:parabolic-14-conic0}--\eqref{eq:parabolic-14-conic1}, so the
proposition applies separately to the two parities.  It follows that the first five
phase states certify quadratic identities in the normalized
conic coordinates.  The quotient identities
\eqref{eq:parabolic-14-Qproducts} have degree at most $4$ in those coordinates:
$u_{j-1}$ and $u_{j+1}$ are affine functions of a consecutive scalar pair,
and $z_j=u_ju_{j+1}-a$ is quadratic.  Thus the first nine states of each
parity certify the corresponding quartic identities.

The degree in this assertion is measured \emph{after substituting} all
auxiliary coordinates.  In particular, degree two in the variables $z_j$
need not mean degree two in the conic coordinates.  The finite certificate
concerns identities along the specified normalized phase orbit; finding an
isolated return, proving integrality or nonvanishing, and imposing
reciprocal coefficient compatibility remain separate questions.

\begin{remark}[the surface parameter is a recurrence coordinate]
\label{rem:parabolic-surface-parameter}
In the $(1,4)$ notation above, the half-state identities imply, for
$j\ge1$,
\begin{equation}\label{eq:parabolic-14-surface-identity}
  aC^{j-1}A^{(j)}(X_j)
  +B^{(j)}(Y_j)
  -a^2C^{2j-1}
  =Y_{j-1}X_jY_j.
\end{equation}
Indeed, the left side equals
\[
  aC^{j-1}(X_j+aC^j)+X_{j-1}X_j-a^2C^{2j-1}
  =X_j(X_{j-1}+aC^{j-1}),
\]
and \eqref{eq:parabolic-14-point-products} at index $j-1$ gives
$X_{j-1}+aC^{j-1}=Y_{j-1}Y_j$.  Thus, after the natural weighting of the
surface, its bilinear parameter is not external to the orbit: it is the
preceding term $Y_{j-1}$ of the same C-finite coordinate sequence.
\end{remark}

\begin{example}[a genuinely dilating $(1,4)$ orbit]
\label{ex:parabolic-14-dilating}
Take
\[
  A(T)=T+2,
  \qquad
  B(T)=T^4+T^3+2T^2+2T+2.
\]
Then $C=2$, so the coefficient orbit is genuinely dilating rather than
periodic.  Starting with $u_0=1$, $u_1=4$, the invariant is $K=8$ and
\eqref{eq:parabolic-14-affine-recurrence} gives
\[
  1,\ 4,\ 30,\ 474,\ 3761,\ 59700,\ 473838,\ldots .
\]
Here $\tau=126$, so
\[
  u_{j+6}-127u_{j+4}+127u_{j+2}-u_j=0.
\]
The corresponding actual half-state coordinates begin
\[
  (X_j,Y_j)=(2,1),\ (236,4),\ (56872,60),\ldots,
\]
and the individual quotient chain begins
\[
  (2,1)\longrightarrow(2,4)\longrightarrow(236,4)
  \longrightarrow(236,60)\longrightarrow(56872,60)\longrightarrow\cdots.
\]
For example,
\[
  \widehat B(T)=T^4+2T^3+4T^2+4T+8,
  \qquad
  \widehat B(4)=472=2\cdot236,
\]
while the next quartic state is
\[
  T^4+2T^3+8T^2+16T+32
\]
and its value at $60$ is $236\cdot56872$.  Thus the C-finite reduction is
not restricted to signed-unit coefficient cycles.
\end{example}

\begin{remark}[two different spectral classifications]
\label{rem:parabolic-two-spectra}
The adjective \emph{parabolic} in Proposition~\ref{prop:parabolic-blocks}
refers to the coefficient dynamics and is forced by the degree product
$nm=4$.  After the geometric gauge, the normalized point dynamics has its
own two-step trace $\tau=\alpha_0\alpha_1-2$.  It may therefore be
elliptic, parabolic, or hyperbolic according to $|\tau|<2$, $|\tau|=2$,
or $|\tau|>2$, while remaining linear and C-finite in all three cases.
\end{remark}

\subsection{Periodic local sieves on integral recurrence states}
\label{sec:periodic-local-sieves}

Finite algebraic identity certificates have a complementary arithmetic
counterpart when a recurrence has an integral invertible state model.
Reduction modulo an integer then gives exact periodic index sets for local
obstructions.  This concerns the chosen integral state; the rational
normalized recurrences above need not themselves have integral coefficients.

\begin{theorem}[periodic local-obstruction sieve]
\label{thm:periodic-affine-sieve}
Let $d\ge1$, let $A\in\operatorname{GL}_d(\Z)$ and $b\in\Z^d$, and let
\[
 v_{i+1}=Av_i+b,\qquad v_0\in\Z^d,\qquad i\ge0.
\]
For each $r\in\{1,\ldots,s\}$ choose
$\kappa_r\in\Z[T_1,\ldots,T_d]$, a modulus $M_r\ge2$, and a set of
forbidden residues $\mathcal B_r\subseteq\Z/M_r\Z$.
Let $h_r$ be the period of the state orbit $v_i\pmod{M_r}$.
Then the index set
\begin{equation}\label{eq:affine-obstruction-indices}
 \mathcal I_r\defeq
 \{i\ge0:\kappa_r(v_i)\bmod M_r\in\mathcal B_r\}
\end{equation}
is a union of residue classes modulo $h_r$.
The complete pattern of these finitely many local tests is periodic modulo
\begin{equation}\label{eq:combined-local-period}
 H\defeq\operatorname{lcm}(h_1,\ldots,h_s).
\end{equation}
In particular, if membership in $\mathcal B_r$ certifies the inadmissibility
of the $r$th candidate, and all candidates are excluded on one residue
class modulo $H$, then they are simultaneously excluded at infinitely many
positions of the same orbit.
\end{theorem}

\begin{proof}
For each $M\ge2$, the reduction of $A$ lies in
$\operatorname{GL}_d(\Z/M\Z)$, because $\det A\in\{1,-1\}$.
Consequently $v\mapsto Av+b$ permutes the finite state space
$(\Z/M\Z)^d$, with inverse $v\mapsto A^{-1}(v-b)$.
Every state orbit is therefore periodic from its first term, with no
transient segment.  Since a polynomial with integer coefficients respects
congruences, $\kappa_r(v_i)\pmod{M_r}$ depends only on $i\pmod{h_r}$.
This proves the residue-class description; taking the least common
multiple proves the simultaneous assertion.
\end{proof}

The third-order recurrence that arises on fixed Pell/Vieta branches is an
immediate special case.  Suppose
\begin{equation}\label{eq:vieta-third-order}
 x_{i+3}=C(x_{i+2}-x_{i+1})+x_i,\qquad i\ge0,
\end{equation}
where $C\in\Z$ and $x_0,x_1,x_2\in\Z$.
Put
\begin{equation}\label{eq:vieta-state-vector}
 S_i\defeq\begin{pmatrix}x_i\\x_{i+1}\\x_{i+2}\end{pmatrix},
 \qquad
 A_C\defeq\begin{pmatrix}0&1&0\\0&0&1\\1&-C&C\end{pmatrix}.
\end{equation}
Then $S_{i+1}=A_CS_i$ and
\begin{equation}\label{eq:vieta-transfer-det}
 \det A_C=1.
\end{equation}
For an affine section-transfer candidate
$\kappa_r(X)=\alpha_rX+\beta_r\in\Z[X]$, apply the theorem to the state
polynomial $\alpha_r T_1+\beta_r$.  It follows that every forbidden-residue
condition
\begin{equation}\label{eq:affine-local-obstruction}
 \kappa_r(x_i)\bmod M_r\in\mathcal B_r
\end{equation}
excludes an exactly computable union of orbit-index classes modulo the
state period, and finitely many such conditions combine through
\eqref{eq:combined-local-period}.  Our companion manuscript \emph{Multivariable Congruence Jumping:
Face Transport, Vieta Mutations, and Six-Jump Returns} applies this
criterion to affine section transfers on fixed Pell/Vieta branches.

\begin{remark}[matrix order as a uniform certificate]
\label{rem:matrix-order-certificate}
In the affine case, set
\[
 \widehat A\defeq
 \begin{pmatrix}A&b\\0&1\end{pmatrix}.
\]
The period of each state modulo $M$ divides the order of $\widehat A$ in
$\operatorname{GL}_{d+1}(\Z/M\Z)$.  Thus a proved congruence
$\widehat A^{\ell}\equiv I\pmod M$ is a uniform period certificate.
For a linear state, one can use $A$ directly; in particular,
$A_C^{\ell}\equiv I\pmod M$ certifies periodicity of every solution of
\eqref{eq:vieta-third-order} modulo $M$.
If the exact matrix order is wanted, it is enough in addition to check
$\widehat A^{\ell/p}\not\equiv I\pmod M$ for every prime $p\mid\ell$.
An individual state, or its admissibility predicate, may have a smaller
period.  A finite observed residue pattern without a proved state return
or matrix relation does not provide this certificate.
\end{remark}

The all-moduli conclusion requires the specified integral unimodular
model.  A rational geometric gauge or rational affine linearization does
not establish that hypothesis.  One must first justify an integral model;
alternatively, for a particular modulus one may reduce a rational model
only after checking that its denominators and determinant are invertible
modulo that modulus.  Local admissibility remains a necessary condition
for a section transfer, and does not by itself establish a return.

\subsection{Hyperbolic point growth after periodic coefficient gauges}
\label{sec:hyperbolic-point-growth}

The four-phase gauge removes the changing coefficient shapes.  It need
not linearize the point recurrence.  In fact the degree product again gives the decisive spectral
boundary.

\begin{proposition}[hyperbolic growth obstruction]
\label{prop:hyperbolic-point-growth}
Let $(u_j)$ be an eventually positive solution of a periodic recurrence
\begin{equation}\label{eq:hyperbolic-periodic-polynomial-recurrence}
  u_{j-1}u_{j+1}=F_j(u_j),
  \qquad F_{j+q}=F_j,
\end{equation}
where the leading coefficients of the $F_j$ are positive and the degrees
alternate between positive integers $n,m$ with
\[
  s\defeq nm>4.
\]
Assume that $u_j\to+\infty$.  Put
\begin{equation}\label{eq:hyperbolic-point-xi}
  \xi_+\defeq
  \frac{s-2+\sqrt{s(s-4)}}2>1.
\end{equation}
Then for each parity $r\in\{0,1\}$ there is a constant $C_r>0$ such that
\begin{equation}\label{eq:hyperbolic-point-log-asymptotic}
  \log u_{2k+r}
  =C_r\xi_+^k+o(\xi_+^k).
\end{equation}
In particular $(u_j)$ grows faster than every exponential sequence and is
not C-finite.
\end{proposition}

\begin{proof}
Let $d_j\defeq\deg F_j$ and let $a_j>0$ be the leading coefficient of
$F_j$.  Since the polynomial family is finite and $u_j\to+\infty$,
\[
  \log F_j(u_j)
  =d_j\log u_j+\log a_j+o(1).
\]
Writing $L_j\defeq\log u_j$ therefore gives
\begin{equation}\label{eq:hyperbolic-log-forced}
  L_{j+1}=d_jL_j-L_{j-1}+c_j+o(1),
  \qquad c_j\defeq\log a_j.
\end{equation}
The sequence $(c_j)$ is periodic.  With a period divisible by both $q$
and $2$, the affine recurrence obtained by deleting the $o(1)$ term has a
periodic particular solution: its homogeneous monodromy has eigenvalues
which are positive powers of $\xi_+$ and $\xi_+^{-1}$, and hence does not
have eigenvalue $1$.  Subtract that periodic solution from $(L_j)$.
Grouping two successive steps then gives
\begin{equation}\label{eq:hyperbolic-log-matrix}
  W_{k+1}=HW_k+\varepsilon_k,
  \qquad \varepsilon_k\longrightarrow0,
\end{equation}
where $\det H=1$, $\operatorname{tr}H=s-2$, and the eigenvalues of $H$
are $\xi_+$ and $\xi_+^{-1}$.

In the eigenbasis of $H$, the expanding coordinate has the form
\[
  A_k=\xi_+^k
  \left(A_0+\sum_{i=0}^{k-1}\xi_+^{-i-1}\varepsilon_i^{(+)}\right).
\]
The parenthesized expression has a finite limit $A$.  If $A=0$, the tail
formula and $\varepsilon_i^{(+)}\to0$ imply $A_k\to0$; the contracting
coordinate also tends to $0$.  Then $L_j$ would remain bounded after the
periodic correction, contrary to $u_j\to+\infty$.  Thus $A\ne0$.
The expanding eigenvector has two positive coordinates, so eventual
positivity and divergence force $A>0$.  This proves
\eqref{eq:hyperbolic-point-log-asymptotic}.

Every C-finite sequence has magnitude $O(k^eR^k)$ for some $e$ and $R$,
whereas \eqref{eq:hyperbolic-point-log-asymptotic} gives
$u_{2k+r}=\exp(C_r\xi_+^k+o(\xi_+^k))$.  Hence $(u_j)$ cannot be
C-finite.
\end{proof}

\begin{corollary}[gauge-independent obstruction]
\label{cor:hyperbolic-gauge-obstruction}
Let a canonical monic coefficient chain have positive rational constant
terms and $nm>4$.  Use the rational four-phase gauge of
Proposition~\ref{prop:universal-four-phase-gauge}, and suppose that its
normalized point orbit is eventually positive and tends to $+\infty$.  Then no positive coefficient-periodicizing gauge can make the
point orbit C-finite.
\end{corollary}

\begin{proof}
Lemma~\ref{lem:periodic-shape-gauge} reduces every such gauge to a periodic
polynomial recurrence of the form
\eqref{eq:hyperbolic-periodic-polynomial-recurrence}, so
Proposition~\ref{prop:hyperbolic-point-growth} applies.  Any two positive
coefficient-periodicizing gauges differ only by a positive periodic
multiplier, which preserves C-finiteness in both directions.
\end{proof}

\begin{example}[the mixed-degree $c=5$ chain]
\label{ex:hyperbolic-c5-point-growth}
The companion arithmetic paper~\cite{AlekseyevKhomovsky2026Arithmetic} constructs the surface
\[
  2x^5+x^2+y^3-3y+5=zxy
\]
and proves that every nonzero integral point generates an infinite
canonical quotient chain.  After the first jump, take the four monic phase
representatives
\begin{align*}
  Q_0(T)&=T^3-3T^2+25,\\
  Q_1(T)&=T^5+125T^3+2\cdot5^9,\\
  Q_2(T)&=T^3-12\cdot5^{16}T+8\cdot5^{25},\\
  Q_3(T)&=T^5+2^8 5^{69}T^2+2^{14}5^{116}.
\end{align*}
Choose the gauge in Proposition~\ref{prop:universal-four-phase-gauge} with
$\delta_0=\delta_1=\delta_2=\delta_3=1$.  Direct reciprocal
normalization gives
\[
  \delta_4=2^{13}5^{107},
  \qquad
  \delta_5=2^{36}5^{298},
  \qquad
  \delta_6=2^{167}5^{1383},
  \qquad
  \delta_7=2^{465}5^{3851},
\]
illustrating the supergeometric coefficient scale.  The four periodic
multipliers in \eqref{eq:periodic-shape-point-recurrence} are
\begin{equation}\label{eq:c5-periodic-lambdas}
  \lambda_0=2^3 5^{23},
  \qquad
  \lambda_1=1,
  \qquad
  \lambda_2=1,
  \qquad
  \lambda_3=2^{-13}5^{-107}.
\end{equation}
Thus the coefficient gauge reduces the changing states exactly to four
fixed polynomial quotient rules; it does not make them linear.

For the integral point $(x,y,z)=(1,1,6)$, write the quotient chain after
the first jump as
\[
  x_{-1}=1,
  \qquad
  x_0=8,
  \qquad
  x_1=345,
\]
so that $x_{j-1}x_{j+1}=P_j(x_j)$ with $P_j$ in the four phases above.
The next two terms are
\[
  x_2=611591853125,
  \qquad
  x_3=663080077940370041526812744140625.
\]
For $u_j=x_j/\delta_j$ one therefore has $u_j=x_j$ for $0\le j\le3$.
At phase $3$,
\[
  u_4
  =\frac{Q_3(u_3)}{2^{13}5^{107}u_2}
  >10^{69}.
\]
At the next phase,
\[
  2^3 5^{23}Q_0(t)>\frac12\,2^3 5^{23}t^3
  \qquad(t\ge6).
\]
Since $u_3<10^{33}$, the displayed bound on $u_4$ gives
$\frac12 2^3 5^{23}u_4/u_3>1$, and therefore $u_5>u_4^2$.  Thereafter the same inequality propagates through the
four phases.  Indeed, if $u_j\ge u_{j-1}^2$, then
$u_{j-1}\le u_j^{1/2}$; the phase-$1$ polynomial dominates $u_j^5$, the
phase-$2$ polynomial dominates $u_j^3/2$ once
$u_j\ge2\sqrt6\,5^8$, and the phase-$3$ polynomial dominates
$u_j^5/(2^{13}5^{107})$, for which
\[
  u_j\ge(2^{13}5^{107})^{2/5}
\]
is sufficient to preserve $u_{j+1}\ge u_j^2$.  These thresholds are
already exceeded before the corresponding phases recur.  Hence
\begin{equation}\label{eq:c5-square-growth}
  u_{j+1}>u_j^2
  \qquad(j\ge4).
\end{equation}
The normalized point orbit is therefore superexponential and not
C-finite.  Here $s=15$, so Proposition~\ref{prop:hyperbolic-point-growth}
identifies the sharper two-step logarithmic growth factor as
$(13+\sqrt{165})/2$.  Since any other positive gauge reducing this
coefficient chain to periodic shapes differs only by a periodic multiplier,
the conclusion is independent of the chosen coefficient-periodicizing
gauge.
\end{example}

\begin{remark}[the degree boundary is a pointwise boundary]
\label{rem:degree-boundary-pointwise}
Theorems~\ref{thm:parabolic-cfinite} and
Proposition~\ref{prop:hyperbolic-point-growth} distinguish two uses of the
coefficient gauge.  At $nm=4$ it removes the parabolic dilation and leaves
a two-periodic affine Vieta system, hence C-finite point dynamics.  For
$nm>4$ and positive constant terms,
Proposition~\ref{prop:universal-four-phase-gauge} always reduces the
coefficient chain to four polynomial shapes, but every eventually positive normalized
orbit tending to infinity has hyperbolic logarithmic growth.  Thus the
difference is not the existence of a coefficient gauge; it is the
nonlinear point dynamics that survives after gauging.  Bounded or periodic
hyperbolic point orbits are not excluded by this obstruction.
\end{remark}

\section{Primitive quadratic coefficient dynamics}\label{sec:quadratic-coeff}

For a nonzero integer polynomial $P$, let $\operatorname{pp}(P)$ denote its primitive part, with sign chosen so that its leading coefficient is positive.  On primitive quadratics with nonzero constant terms define the normalized jump-and-swap map
\begin{equation}\label{eq:J-map}
  \mathcal J(A,B)
  \defeq
  \bigl(\operatorname{pp}(\Rcal_{A(0)}B),A\bigr).
\end{equation}
This records the coefficient transformation produced by one quotient jump, followed by primitive normalization and interchange of the two coordinates.

\begin{proposition}[periodic primitive quadratic states]\label{prop:quadratic-periodic}
The map $\mathcal J$ is a bijection.  If a primitive pair $(A,B)$ is periodic under $\mathcal J$, then there are $C>0$, $b,e\in\Z$, and $\varepsilon,\eta\in\{\pm1\}$ such that
\begin{equation}\label{eq:periodic-quadratic-form}
  A(T)=T^2+bT+\varepsilon C,
  \qquad
  B(T)=T^2+eT+\eta C.
\end{equation}
Conversely every pair of the form \eqref{eq:periodic-quadratic-form} is periodic.  Its period divides $2$ when $\varepsilon=\eta$, and divides $4$ when $\varepsilon=-\eta$; more precisely, in the latter case the period is $2$ exactly when $b=e=0$ and is $4$ otherwise.

In particular, there are no genuinely nonmonic periodic primitive quadratic coefficient states.
\end{proposition}

\begin{proof}
For a quadratic polynomial, $\Rcal_c^2P=c^2P$.  Taking primitive parts therefore gives
\[
  \mathcal J^{-1}(A,B)
  =\bigl(B,\operatorname{pp}(\Rcal_{B(0)}A)\bigr),
\]
so $\mathcal J$ is bijective.

Fix a prime $p$ and write
\[
  A=aT^2+bT+c,
  \qquad
  B=dT^2+eT+f,
\]
with
\[
  \alpha=v_p(a),\quad \gamma=v_p(c),\quad
  \delta=v_p(d),\quad \phi=v_p(f).
\]
Set
\[
  \rho\defeq\gamma-\alpha,
  \qquad
  \sigma\defeq\phi-\delta,
  \qquad
  \Delta\defeq\rho-\sigma.
\]
Let $C_1:=\operatorname{pp}(\Rcal_cB)$, let
$h:=v_p(\cont(\Rcal_cB))$, and put
$\delta_1:=v_p(\operatorname{lc}(C_1))=\phi-h$.
The constant coefficient of $C_1$ has valuation $\delta+2\gamma-h$, so its endpoint-valuation difference is
\[
  \theta_1
  :=v_p(C_1(0))-v_p(\operatorname{lc}(C_1))
  =2\gamma-\sigma.
\]
Its constant coefficient therefore has valuation $\delta_1+\theta_1$.  If
$C_2:=\operatorname{pp}(\Rcal_{C_1(0)}A)$, then the endpoint-valuation difference of $C_2$ is
\[
  \theta_2
  =\alpha+2(\delta_1+\theta_1)-\gamma.
\]
After two applications of $\mathcal J$ the new difference is
\[
  \Delta'
  =\theta_2-\theta_1
  =\alpha-\gamma+2\delta_1+\theta_1
  =\Delta+2\alpha+2\delta_1.
\]
Thus
\begin{equation}\label{eq:delta-monotone}
  \Delta' = \Delta+2\alpha+2\delta_1.
\end{equation}
In particular, $\Delta$ is nondecreasing along the even iterates.  On a periodic orbit it must return to its original value, so equality holds in \eqref{eq:delta-monotone} at every shifted state.  Hence every leading coefficient occurring on the orbit is a $p$-adic unit.  Since $p$ was arbitrary and the leading coefficients are positive, all of them equal $1$.

With the leading coefficients now equal to $1$, one has
$\Delta=\gamma-\phi$.  The first transformed constant has valuation $2\gamma-\phi=\gamma+\Delta$, and after the second jump the first constant has valuation $\gamma+2\Delta$; meanwhile \eqref{eq:delta-monotone} shows that $\Delta$ itself is unchanged.  Periodicity therefore forces $\Delta=0$.  Hence $v_p(|A(0)|)=v_p(|B(0)|)$ for every prime $p$, so $|A(0)|=|B(0)|=C$ and \eqref{eq:periodic-quadratic-form} follows.

For such a pair one checks directly that
\begin{equation}\label{eq:J-periodic-form}
  \mathcal J(A,B)
  =\bigl(T^2+\varepsilon\eta eT+\eta C,A\bigr).
\end{equation}
If $\varepsilon=\eta$, this merely swaps the two polynomials.  If $\varepsilon=-\eta$, two jumps negate both middle coefficients, giving the stated periods.
\end{proof}

\section{Relation-space transport and persistent saturation}\label{sec:persistent-saturation}

The preceding sections concern coefficient orbits and do not require a particular integer point to follow them.  We now return to the pointwise relations transported by a quotient jump.  We first isolate a general rational-rank obstruction for simultaneous separated relations.  It explains why enlarging the list of auxiliary relations can increase the information carried through one jump but cannot sustain a higher-rank controlled system through alternating jumps.  We then specialize to quadratic Vieta orbits and study the genuinely arithmetic phenomenon of primitive saturation.

\subsection{Transported relation spaces and rank collapse}\label{sec:relation-space-rank}

For $n,m\ge1$, let $V_{n,m}$ be the $\mathbb Q$-vector space of separated polynomials
\[
  F(X,Y)=A(X)+B(Y),
  \qquad
  \deg A\le n,\quad \deg B\le m,\quad B(0)=0.
\]
The condition $B(0)=0$ makes the decomposition unique.  Define the side projections
\begin{equation}\label{eq:relation-side-projections}
  \pi_XF\defeq A(X),
  \qquad
  \pi_YF\defeq A(0)+B(Y).
\end{equation}
Thus $\pi_YF$ is precisely the part of $F(x,y)$ visible modulo $x$.

Fix an active relation whose $X$-side is a polynomial $P(X)$ with
\[
  c\defeq P(0)\ne0,
\]
and formally write the corresponding $Y$-jump as $Y'=P(X)/Y$.  For
\[
  C(Y)=\sum_{j=0}^{m}b_jY^j
\]
define
\begin{equation}\label{eq:relation-J-map}
  J_{P,c}^{(m)}C
  \defeq
  b_m c^{m-1}P(X)
  +\sum_{j=0}^{m-1}b_jc^j(Y')^{m-j}.
\end{equation}
The transported relation operator is
\begin{equation}\label{eq:relation-transport-factorization}
  T_Y\defeq J_{P,c}^{(m)}\circ\pi_Y.
\end{equation}
At this stage $T_Y$ is a rational-linear transport operator; for general $c$ no pointwise divisibility assertion is being made.  Its direct interpretation as transport of congruence relations is recorded below under the unit-constant hypothesis $c=\pm1$.

\begin{proposition}[side-projection factorization and two-step rank collapse]\label{prop:relation-rank-collapse}
The map $J_{P,c}^{(m)}$ is injective.  Hence for every rational relation subspace $R\subseteq V_{n,m}$,
\begin{equation}\label{eq:relation-one-step-rank}
  \dim_{\mathbb Q}T_Y(R)=\dim_{\mathbb Q}\pi_Y(R)\le m+1.
\end{equation}
Moreover,
\begin{equation}\label{eq:relation-opposite-projection-rank}
  \pi_X\bigl(T_Y(R)\bigr)\subseteq\mathbb Q P,
\end{equation}
so the $X$-side rank after one $Y$-jump is at most one.  Consequently, if an admissible next $X$-jump is defined by the symmetric construction, then
\begin{equation}\label{eq:relation-two-step-collapse}
  \dim_{\mathbb Q}T_XT_Y(R)\le1.
\end{equation}
The same statement holds with $X$ and $Y$ interchanged.
\end{proposition}

\begin{proof}
If $J_{P,c}^{(m)}C=0$, comparison of the coefficients of $(Y')^m,\ldots,Y'$ in \eqref{eq:relation-J-map} gives
\[
  b_0=\cdots=b_{m-1}=0,
\]
because $c\ne0$.  The remaining term is $b_mc^{m-1}P(X)$, so $b_m=0$ as well.  Thus $J_{P,c}^{(m)}$ is injective, proving \eqref{eq:relation-one-step-rank} from \eqref{eq:relation-transport-factorization}.

The $X$-part of \eqref{eq:relation-J-map} is always $b_mc^{m-1}P(X)$, which proves \eqref{eq:relation-opposite-projection-rank}.  A subsequent $X$-transport factors symmetrically through the $X$-side projection of its input space.  Its image therefore has dimension equal to that projected dimension, which is at most one by \eqref{eq:relation-opposite-projection-rank}.
\end{proof}

When $c=\pm1$, the formal operator above is also the ordinary pointwise relation transport.  Indeed, if $F(x,y)\equiv0\pmod{xy}$ and $y'=P(x)/y$ is integral, then modulo $x$ multiplication by $(y')^m$ and the congruence $yy'\equiv c$ give the reciprocal expression in \eqref{eq:relation-J-map}; modulo $y'$ every term in \eqref{eq:relation-J-map} is divisible by $y'$.  Since $P(x)\equiv c\pmod x$, one has $\gcd(x,y')=1$, and hence the transported expression is divisible by $xy'$.

\begin{remark}[saturation does not restore rational rank]\label{rem:relation-saturation-rank}
Let $E'$ be the full integral evaluation-relation lattice at the new point and let $L\subseteq E'$ be an integral transported lattice.  Its saturation inside $E'$ is
\[
  \operatorname{Sat}_{E'}(L)
  =E'\cap(L\otimes_{\mathbb Z}\mathbb Q).
\]
Therefore
\[
  \operatorname{Sat}_{E'}(L)\otimes_{\mathbb Z}\mathbb Q
  =L\otimes_{\mathbb Z}\mathbb Q.
\]
Primitive normalization can enlarge the integral lattice by finite index, but it cannot create a rational direction lost under the side projection.  In particular, the two-step bound \eqref{eq:relation-two-step-collapse} survives arbitrary integral saturation.
\end{remark}

\begin{remark}[companion-family rank preservation]\label{rem:relation-companion-rank}
If a family of relations has a common $X$-side, then its difference space has zero $X$-part, so $\pi_Y$ is injective on that difference space.  Equation~\eqref{eq:relation-one-step-rank} then shows that one reciprocal jump preserves its rational rank.  Thus the familiar rational-rank preservation for differences of companion surfaces is the injective-projection special case of Proposition~\ref{prop:relation-rank-collapse}; the new content is that alternating once more forces rank at most one.
\end{remark}

\begin{example}[sparse and full degree-$(5,3)$ supports]\label{ex:relation-rank-53}
Take the active relation
\[
  F_*(X,Y)=3X^5+X+Y^3+1,
  \qquad
  P(X)=3X^5+X+1,
\]
so $c=1$.  On the sparse basis $(X^5,X,Y^3,1)$, a coefficient vector $(\alpha,\beta,\gamma,\delta)$ is transported by
\begin{equation}\label{eq:relation-sparse-map}
  (\alpha,\beta,\gamma,\delta)
  \longmapsto
  (3\gamma,\gamma,\delta,\gamma),
\end{equation}
which has rank $2$; the active vector $(3,1,1,1)$ is fixed.  If instead the full reciprocal-closed degree-$(5,3)$ support
\[
  (X^5,X^4,X^3,X^2,X,Y^3,Y^2,Y,1)
\]
is used, then
\begin{equation}\label{eq:relation-full-map}
  (a_5,a_4,a_3,a_2,a_1,b_3,b_2,b_1,d)
  \longmapsto
  (3b_3,0,0,0,b_3,d,b_1,b_2,b_3),
\end{equation}
which has rank $4$.  Thus filling the cubic support increases the one-step information rank from $2$ to $4$, exactly as \eqref{eq:relation-one-step-rank} predicts.  Nevertheless Proposition~\ref{prop:relation-rank-collapse} forces every subsequent alternating $X$-transport of either image to have rational rank at most $1$.
\end{example}

\subsection{Local saturation on quadratic Vieta orbits}\label{sec:local-quadratic-saturation}

We now specialize from this rational rank obstruction to integral saturation along a single quadratic Vieta orbit.  The terminology elliptic, parabolic, and hyperbolic below refers to the linear part of the \emph{point} recurrence, and should not be confused with the degree-product trichotomy of Section~\ref{sec:iterated-gauge}.

Let
\[
  P(T)=T^2+pT+c\in\Z[T],
  \qquad c\ne0,
\]
and fix an integer $K$.  Suppose that a nonzero integer point $(x,y)$ lies on
\begin{equation}\label{eq:ps-base-conic}
  P(x)+P(y)-c=Kxy.
\end{equation}
The quotient jump is integral and replaces $y$ by
\begin{equation}\label{eq:ps-vieta-jump}
  y'\defeq\frac{P(x)}y=Kx-p-y.
\end{equation}
Thus the homogeneous part of the point recurrence has transfer matrix
\[
  V_K=\begin{pmatrix}K&-1\\1&0\end{pmatrix},
\]
with the usual elliptic, parabolic, and hyperbolic regimes $|K|<2$, $|K|=2$, and $|K|>2$.

For integers $u,v$, consider an auxiliary common-constant polynomial
\begin{equation}\label{eq:ps-auxiliary-polynomial}
  \widetilde P(T)=P(T)+T(uT+v).
\end{equation}
The auxiliary surface $P(x)+\widetilde P(y)-c=\widetilde Kxy$ has an
integer level $\widetilde K=K+(uy+v)/x$ precisely when
\begin{equation}\label{eq:ps-source-condition}
  x\mid uy+v.
\end{equation}
For all auxiliary relations below, use degree-two reciprocal transport
even when $u=-1$ lowers the actual degree:
\[
  \widetilde C(T)\defeq c^{-1}T^2\widetilde P(c/T)
  =T^2+(p+v)T+(1+u)c,
\]
so that
\begin{equation}\label{eq:ps-transformed-difference}
  \widetilde C(T)-P(T)=vT+uc.
\end{equation}
Suppose $vT+uc=hG(T)$ with $h>1$ and $G\in\mathbb Z[T]$.  The source condition gives $x\mid hG(y')$, since $vy'+uc\equiv y'(uy+v)\equiv0\pmod x$ by $yy'=P(x)\equiv c\pmod x$.  If in fact
\begin{equation}\label{eq:ps-saturation}
  x\mid G(y'),
\end{equation}
we say that the jump exhibits \emph{$h$-saturation}; it is genuinely nonunit when $\gcd(x,h)>1$.

The amount of cancellation at one prime has a simple rank-two lattice description.  Fix a prime $\ell$, put
\[
  e\defeq v_\ell(x),
\]
and define the local relation lattice
\begin{equation}\label{eq:ps-relation-lattice}
  L_{\ell,e}(y)
  \defeq
  \{(a,b)\in\mathbb Z_\ell^2: ay+b\equiv0\pmod{\ell^e}\}.
\end{equation}
Thus
\[
  L_{\ell,e}(y)
  =\mathbb Z_\ell(1,-y)+\mathbb Z_\ell(0,\ell^e),
  \qquad
  [\mathbb Z_\ell^2:L_{\ell,e}(y)]=\ell^e.
\]
For a primitive vector $z=(a,b)\in\mathbb Z_\ell^2$, write
\begin{equation}\label{eq:ps-tent-depth}
  \lambda_{\ell,e,y}([a:b])
  \defeq
  \max\{0,e-v_\ell(ay+b)\}.
\end{equation}
Here $v_\ell(0)=+\infty$.  The quantity in \eqref{eq:ps-tent-depth} is the least $k\ge0$ for which $\ell^kz\in L_{\ell,e}(y)$; equivalently it is the depth defect of the projective direction $[a:b]$ from the root ball centered at $[1:-y]$ modulo $\ell^e$.

\begin{proposition}[exact local cancellation depth]\label{prop:ps-local-cancellation}
Assume the source condition \eqref{eq:ps-source-condition} and $(u,v)\ne(0,0)$.  Let
\[
  T_c(a,b)\defeq(b,ca),
\]
and write
\begin{equation}\label{eq:ps-raw-factorization}
  T_c(u,v)=(v,cu)=\ell^s z,
\end{equation}
where $z=(a,b)$ is primitive over $\mathbb Z_\ell$.  Then
\[
  T_cL_{\ell,e}(y)\subseteq L_{\ell,e}(y')
\]
and the largest exponent $r$, $0\le r\le s$, for which one may cancel $\ell^r$ from the raw transformed relation while retaining a valid relation modulo $x$ is
\begin{equation}\label{eq:ps-max-cancellation}
  r_{\max}
  =s-\lambda_{\ell,e,y'}([z]).
\end{equation}
In particular, the entire $\ell$-part of the raw content can be cancelled if and only if
\begin{equation}\label{eq:ps-full-local-cancellation}
  \lambda_{\ell,e,y'}([z])=0,
\end{equation}
that is, if and only if $ay'+b\equiv0\pmod{\ell^e}$.
\end{proposition}

\begin{proof}
Since $yy'=P(x)\equiv c\pmod{x}$, one has
\[
  vy'+cu\equiv y'(uy+v)\equiv0\pmod{\ell^e},
\]
which proves the lattice inclusion.  After cancellation of $\ell^r$, the transformed coefficient vector is $\ell^{s-r}z$.  By \eqref{eq:ps-tent-depth}, this vector belongs to $L_{\ell,e}(y')$ if and only if
\[
  s-r\ge\lambda_{\ell,e,y'}([z]).
\]
The uncancelled vector with $r=0$ is already in the target lattice, so the right-hand side is at most $s$.  Maximizing $r$ gives \eqref{eq:ps-max-cancellation}.
\end{proof}

The projective transport induced by $T_c$ is especially rigid:
\begin{equation}\label{eq:ps-projective-involution}
  F_c([a:b])\defeq[b:ca],
  \qquad
  F_c^2=\operatorname{id},
\end{equation}
because $T_c^2=cI$.  Thus an iteratively transported auxiliary direction alternates between at most two projective directions.  In the fixed-auxiliary problem below the source direction is constant, its single image under $F_c$ is constant, and only the root balls and their depths vary around the point orbit.

\begin{corollary}[finite criterion on a periodic Vieta orbit]\label{cor:ps-periodic-finite-test}
Let $(x_n)_{n\in\mathbb Z}$ be a nonzero integer Vieta orbit of period $N$, with the $n$th jump based at $(x_n,x_{n-1})$ and ending at $x_{n+1}$.  Fix a nonzero auxiliary relation vector $q=(u,v)$.  For a prime $\ell$, write
\[
  q=\ell^t z^-,
  \qquad
  T_cq=\ell^s z^+,
\]
with $z^-,z^+$ primitive, and put $e_n=v_\ell(x_n)$.  Then the same auxiliary surface is valid at every point of the periodic orbit if and only if
\begin{equation}\label{eq:ps-periodic-source-test}
  \lambda_{\ell,e_n,x_{n-1}}([z^-])\le t
  \qquad(0\le n<N)
\end{equation}
for every prime $\ell$.  The entire local content of the transformed relation can be cancelled at every jump if and only if, in addition,
\begin{equation}\label{eq:ps-periodic-target-test}
  \lambda_{\ell,e_n,x_{n+1}}([z^+])=0
  \qquad(0\le n<N)
\end{equation}
for every $\ell$.  Hence indefinite full saturation on a periodic point orbit is a finite intersection of explicit $\ell$-adic projective-ball conditions.
\end{corollary}

\begin{proof}
Condition \eqref{eq:ps-periodic-source-test} is exactly the statement that $q\in L_{\ell,e_n}(x_{n-1})$ for every source point.  Proposition~\ref{prop:ps-local-cancellation} gives \eqref{eq:ps-periodic-target-test} for full cancellation at the corresponding target.  Periodicity reduces both families to the displayed finite range.
\end{proof}

\subsection{Persistence across point-orbit regimes}\label{sec:persistence-regimes}

The six-cycle used below can be classified completely rather than by exhibiting one auxiliary relation.

\begin{proposition}[all fully saturated auxiliary relations on the irreducible six-cycle]\label{prop:ps-six-cycle-classification}
Let $r\ge2$, let
\[
  P_r(T)=T^2-3r^2,
\]
and let $a_0=1$, $a_1=2$, $a_{n+1}=a_n-a_{n-1}$, so that
\[
  1,2,1,-1,-2,-1,1,\ldots
\]
is six-periodic.  Put
\[
  (x_n,y_n)=(ra_n,ra_{n-1}).
\]
Consider a nonzero fixed auxiliary difference $uT+v$ and
\[
  \widetilde P(T)=P_r(T)+T(uT+v).
\]
Then this auxiliary surface is valid at all six source points only if $v=rw$ for some $w\in\mathbb Z$, and in that notation it is valid at all six points if and only if
\begin{equation}\label{eq:ps-six-source-parity}
  u\equiv w\pmod2.
\end{equation}
Let
\[
  d\defeq\gcd(w,3u)>0.
\]
The transformed relation can be divided by its full content and remains valid at every jump of the six-cycle if and only if, in addition to \eqref{eq:ps-six-source-parity},
\begin{equation}\label{eq:ps-six-classification}
  \gcd\left(\frac wd,r\right)=1,
  \qquad
  \frac{w+3u}{d}\equiv0\pmod2.
\end{equation}
Under these conditions the full content is $rd$, and the primitive transformed relation is
\begin{equation}\label{eq:ps-six-primitive-relation}
  G(T)=\frac wdT-\frac{3ru}{d}.
\end{equation}
The auxiliary polynomial itself is primitive precisely when
\begin{equation}\label{eq:ps-six-aux-primitivity}
  \gcd(1+u,rw,3r^2)=1.
\end{equation}
\end{proposition}

\begin{proof}
The source condition is
\[
  ra_n\mid ur a_{n-1}+v.
\]
The four indices with $a_n=\pm1$ force $r\mid v$; write $v=rw$.  After division by $r$, the only nontrivial conditions occur at $a_n=\pm2$ and both reduce to \eqref{eq:ps-six-source-parity}.

The raw transformed difference is
\[
  rwT-3r^2u
  =r\bigl(wT-3ru\bigr).
\]
Write $w=dw_0$ and $3u=du_0$, with $\gcd(w_0,u_0)=1$.  Its content is
\[
  r\gcd(w,3ru)
  =rd\gcd(w_0,r).
\]
For the primitive transformed relation to remain divisible by every source modulus, it must in particular remain divisible by $r$ at the four positions $|a_n|=1$.  Since $\gcd(w_0,u_0)=1$, this is equivalent to $\gcd(w_0,r)=1$, which is the first condition in \eqref{eq:ps-six-classification}.  The full content is then $rd$ and the primitive relation is \eqref{eq:ps-six-primitive-relation}.

Evaluating at the next coordinate gives
\[
  G(ra_{n+1})
  =r\frac{wa_{n+1}-3u}{d}.
\]
Again the four positions $a_n=\pm1$ are automatic.  At the two positions $a_n=\pm2$ the conditions are
\[
  \frac{w-3u}{d}\equiv0\pmod2,
  \qquad
  \frac{w+3u}{d}\equiv0\pmod2.
\]
They are equivalent because their difference is $2w/d$.  This proves the second condition in \eqref{eq:ps-six-classification}.  Finally, \eqref{eq:ps-six-aux-primitivity} is the coefficient gcd of
\[
  \widetilde P(T)=(1+u)T^2+rwT-3r^2.
\]
\end{proof}

The following theorem places three persistence mechanisms side by side.  Complete constructions and proofs are given in Appendix~\ref{app:persistence-proofs}.

\begin{theorem}[persistent saturation across point-orbit regimes]\label{thm:persistent-saturation-regimes}
The quadratic Vieta dynamics admits the following three types of genuinely nonunit saturation.
\begin{enumerate}[label=\textup{(\alph*)}]
\item \textbf{Periodic elliptic persistence.}
For every integer $r\ge2$ with $3\nmid r$, the irreducible conic
\[
  x^2+y^2-xy-3r^2=0
\]
has a six-step integer Vieta orbit $(x_n,y_n)=(ra_n,ra_{n-1})$, where
\[
  a_0=1,\qquad a_1=2,\qquad a_{n+1}=a_n-a_{n-1}.
\]
Proposition~\ref{prop:ps-six-cycle-classification} classifies all fixed auxiliary relations that fully saturate on this cycle.  In particular, taking $u=w=2$ gives
\[
  P_r(T)=T^2-3r^2,
  \qquad
  \widetilde P_r(T)=3T^2+2rT-3r^2,
\]
with normalized transformed difference
\[
  2r(T-3r).
\]
Full saturation persists at every jump, the removed content $2r$ is a nonunit at every source modulus, and the assumption $3\nmid r$ is exactly what makes this particular auxiliary polynomial primitive.

\item \textbf{Arbitrarily long parabolic persistence.}
For every $L\ge1$ there exist a nonsingular irreducible conic of the form
\[
  P(x)+P(y)-P(0)=2xy
\]
with an unbounded nonperiodic integer Vieta orbit and one fixed auxiliary quadratic polynomial that gives genuine nonunit $2$-saturation on $L$ consecutive jumps.

\item \textbf{Arbitrarily long persistence on one fixed hyperbolic orbit.}
Let
\[
  P(T)=T^2-45,
  \qquad
  x^2+y^2-45=3xy,
\]
and let $w_j=L_{2j+1}=1,4,11,29,76,\ldots$ be the odd-index Lucas subsequence, equivalently
\[
  w_{j+1}=3w_j-w_{j-1}.
\]
The points
\[
  (x_j,y_j)=(3w_j,3w_{j-1})
\]
form a fixed unbounded nonperiodic integer Vieta orbit.  For every $L\ge1$ there is an integer $a=a_L$ such that the single auxiliary polynomial
\[
  \widetilde P_a(T)=(1+a)T^2+3T-45
\]
gives genuine nonunit $3$-saturation at each of the first $L$ jumps.  Its normalized transformed difference is
\[
  3(T-15a).
\]
No fixed $a$ in this family gives saturation at infinitely many jumps of this orbit.
\end{enumerate}
\end{theorem}

\begin{example}[three-step hyperbolic persistence]\label{ex:ps-hyperbolic-three}
For $L=3$ in Theorem~\ref{thm:persistent-saturation-regimes}\textup{(c)}, one may take $a=195$.  Then
\[
  \widetilde P(T)=196T^2+3T-45,
  \qquad
  \widetilde C(T)-P(T)=3(T-2925).
\]
The first three source moduli are $12,33,87$, and the corresponding next coordinates are $33,87,228$.  After removing the nonunit factor $3$, the primitive factor $G(T)=T-2925$ satisfies
\[
  G(33)=-2892=-241\cdot12,
  \qquad
  G(87)=-2838=-86\cdot33,
  \qquad
  G(228)=-2697=-31\cdot87.
\]
Thus the strengthened relation persists for these three successive jumps on the same hyperbolic orbit.  The theorem shows that the window can be made arbitrarily long, while no fixed choice of $a$ persists indefinitely.
\end{example}

\begin{corollary}[no finite persistence certificate on the fixed hyperbolic orbit]\label{cor:no-finite-saturation-horizon}
There is no integer $H$ such that success of one auxiliary relation through the first $H$ jumps of the fixed hyperbolic orbit in Theorem~\ref{thm:persistent-saturation-regimes}\textup{(c)} forces that relation to persist indefinitely.
\end{corollary}

\begin{proof}
Given $H$, choose the auxiliary polynomial supplied by Theorem~\ref{thm:persistent-saturation-regimes}\textup{(c)} with $L=H$.  It succeeds through those $H$ jumps, whereas no fixed member of the family succeeds infinitely often.
\end{proof}

The corollary is a pointwise analogue of Remark~\ref{rem:hyperbolic-no-uniform-horizon}.  The mechanisms are different: Proposition~\ref{prop:hyperbolic-delayed} delays failure of coefficient integrality by moving along a valuation recurrence, while Theorem~\ref{thm:persistent-saturation-regimes}\textup{(c)} keeps the base hyperbolic point orbit fixed and changes only the auxiliary relation.  In both cases, however, an arbitrarily long successful prefix does not certify global persistence.

\section{Conclusion and further directions}\label{sec:outlook}

For each fixed degree pair $(n,m)$, alternating reciprocal normalization gives a finite-dimensional rational return map on coefficient space.  Its four-step rational gauge closure is universal, but its integral orbits exhibit a sharp qualitative trichotomy after monic reduction.  The low-product cases are finite order, the degree product $nm=4$ produces parabolic dilations, and $nm>4$ produces hyperbolic valuation dynamics.  Every positive-constant canonical monic chain admits the explicit rational four-phase gauge of Proposition~\ref{prop:universal-four-phase-gauge}. At the parabolic degree boundary a geometric gauge further linearizes the point orbit: in all three patterns $(2,2)$, $(1,4)$, and $(4,1)$, every nondegenerate compatible rational orbit reduces to a two-periodic affine recurrence and hence to C-finite phase-coordinate sequences. On each fixed invariant conic, degree-$h$ identities are certified by $2h+1$ initial phase states; an integral unimodular state model also turns local congruence obstructions into exact periodic sets of indices.  Beyond the boundary, periodicizing the coefficient shapes no longer linearizes the points: for $nm>4$, every eventually positive normalized orbit that tends to infinity has hyperbolic logarithmic growth and is not C-finite, independently of the positive coefficient-periodicizing gauge.  Endpoint defects provide the exact local admissibility constraints, while the hyperbolic forward and bi-infinite integral loci are cut out by irrational eigencones together with finitely computable valley conditions.  Arbitrarily delayed obstructions show that these global criteria cannot be replaced by any uniform finite successful prefix.

The pointwise layer has two complementary kinds of rigidity.  Proposition~\ref{prop:relation-rank-collapse} shows that separated auxiliary-relation transport factors through side projections: one jump preserves exactly the visible support rank, but two alternating jumps reduce every transported rational relation space to dimension at most one.  Remark~\ref{rem:relation-saturation-rank} shows that primitive saturation can change integral index but cannot restore a lost rational direction.  In the quadratic specialization, Proposition~\ref{prop:ps-local-cancellation} identifies the exact uncancellable exponent with a projective tent depth.  On periodic Vieta orbits this gives a finite projective-ball criterion, and Proposition~\ref{prop:ps-six-cycle-classification} completely classifies the fixed auxiliary relations that fully saturate on the basic irreducible six-cycle.  Beyond the periodic setting, nonunit saturated relations can persist through arbitrarily long finite windows on unbounded parabolic and hyperbolic orbits.  Thus finite-time admissibility and global persistence separate at both the coefficient and pointwise levels.

Several extensions are natural from the dynamical viewpoint.  Over rings of integers in number fields, prime-ideal valuations should replace the scalar $p$-adic data and may lead to higher-dimensional invariant cones.  Even over $\mathbb Z$, it would be useful to describe invariant subvarieties and periodic points of the rational block map beyond the quadratic locus and to understand the valuation recurrence systematically as a non-Archimedean or tropical factor of the coefficient dynamics.  Proposition~\ref{prop:relation-rank-collapse} also identifies a boundary for higher-rank pointwise dynamics: ordinary separated transport with one active quotient cannot sustain more than one rational relation direction through two alternating jumps.  Any genuinely higher-rank extension must therefore inject new rational directions, for example through mixed monomials, simultaneous active quotient structures, or additional companion data not factored through a single side projection.

\section*{Acknowledgements}

The authors used ChatGPT to assist with exploratory calculations, checking derivations, and improving the exposition. The authors take full responsibility for the mathematical arguments and all other content of the paper.

\printbibliography

\clearpage
\appendix
\refstepcounter{section}
\section*{Appendix \thesection: Proofs for persistent saturation}\label{app:persistence-proofs}

This appendix proves the parabolic and hyperbolic parts of Theorem~\ref{thm:persistent-saturation-regimes}.  The periodic part follows directly from Proposition~\ref{prop:ps-six-cycle-classification} with $u=w=2$.  The constructions below use only the quadratic setup of Section~\ref{sec:persistent-saturation}; in particular, the base point dynamics is always the Vieta recurrence \eqref{eq:ps-vieta-jump}.

\subsection{Arbitrarily long parabolic persistence}\label{appsec:parabolic-persistence}

\begin{proof}[Proof of Theorem~\ref{thm:persistent-saturation-regimes}\textup{(b)}]
Fix $L\ge1$.  Choose a positive integer $S$ such that
\begin{equation}\label{eq:app-S-mod4}
  S\equiv2\pmod4
\end{equation}
and, for every odd prime $q$ dividing $m^2-n^2$ for some $1\le m<n\le L$, the residue $-S$ is a quadratic nonresidue modulo $q$.  The Chinese remainder theorem supplies such an $S$.  Put
\[
  D_n=n^2+S.
\]
If an odd prime divided both $D_m$ and $D_n$, then it would divide $m^2-n^2$ and make $-S$ a quadratic residue, contrary to the choice of $S$.  Condition \eqref{eq:app-S-mod4} also bounds the common $2$-part, so
\begin{equation}\label{eq:app-D-gcd}
  \gcd(D_m,D_n)\le2
  \qquad(1\le m<n\le L).
\end{equation}
Consequently the congruences
\begin{equation}\label{eq:app-W-crt}
  W\equiv4n-2\pmod{D_n},
  \qquad 1\le n\le L,
\end{equation}
are compatible: the difference of two right-hand sides is divisible by $4$, whereas every pairwise gcd in \eqref{eq:app-D-gcd} is at most $2$.  Fix one solution $W$.

Set
\[
  c=4(1+4S),
  \qquad
  P(T)=T^2-4T+c,
  \qquad
  u_n=2D_n.
\]
A direct calculation gives
\[
  u_{n+1}=2u_n-u_{n-1}+4
\]
and
\[
  (u_n-u_{n-1})^2-4(u_n+u_{n-1})+c=0.
\]
Thus $(u_n,u_{n-1})$ lies on the fixed conic
\[
  P(x)+P(y)-c=2xy.
\]
Its projective determinant is $-16$, so it is nonsingular and irreducible, and $u_n\sim2n^2$ makes the orbit unbounded and nonperiodic.

Use the fixed auxiliary polynomial
\[
  \widetilde P_W(T)=P(T)+2T(T+W)=3T^2+(2W-4)T+c.
\]
For $1\le n\le L$, write
\[
  W=4n-2+q_nD_n.
\]
At the source point $(x,y)=(u_n,u_{n-1})$,
\[
  2(y+W)=2D_n(q_n+2),
\]
so $x=2D_n$ satisfies the auxiliary source condition.  The normalized reciprocal is
\[
  \widetilde C_W(T)=T^2+(2W-4)T+3c,
\]
and hence
\[
  \widetilde C_W(T)-P(T)=2F_W(T),
  \qquad
  F_W(T)=WT+c.
\]
The next quotient is $y'=u_{n+1}=2D_{n+1}$, and substitution of \eqref{eq:app-W-crt} gives
\[
  F_W(u_{n+1})
  =2D_n\bigl(2(2n+3)+q_nD_{n+1}\bigr).
\]
Therefore $x\mid F_W(y')$ on all $L$ selected jumps.  Since $2\mid x$, removing the displayed factor $2$ is genuinely nonunit.
\end{proof}

\begin{remark}\label{rem:parabolic-fixed-W-finite}
For fixed $S$ and $W$, the source condition above is equivalent to
\[
  D_n\mid W-(4n-2).
\]
For sufficiently large $n$ the nonzero integer on the right has absolute value smaller than $D_n=n^2+S$.  Hence the construction gives arbitrarily long finite persistence, not an infinite ladder for one fixed $W$.
\end{remark}

\subsection{Arbitrarily long persistence on one fixed hyperbolic orbit}\label{appsec:hyperbolic-persistence}

Let $(L_n)_{n\ge0}$ denote the Lucas sequence
\[
  L_0=2,
  \qquad L_1=1,
  \qquad L_{n+2}=L_{n+1}+L_n,
\]
and put
\[
  w_j=L_{2j+1}=1,4,11,29,76,199,521,\ldots.
\]
Then
\begin{equation}\label{eq:app-w-recurrence}
  w_{j+1}=3w_j-w_{j-1},
  \qquad
  w_{j-1}w_{j+1}=w_j^2-5.
\end{equation}

\begin{lemma}[compatible Lucas congruences]\label{lem:app-lucas-crt}
For every $L\ge1$, the system
\begin{equation}\label{eq:app-lucas-crt}
  5a\equiv w_{j+1}\pmod{w_j},
  \qquad 1\le j\le L,
\end{equation}
has an integer solution $a$.
\end{lemma}

\begin{proof}
Let $(F_n)$ denote the Fibonacci sequence, defined by
$F_0:=0$, $F_1:=1$, and $F_{n+2}:=F_{n+1}+F_n$.
The identity
\[
  L_n^2-5F_n^2=4(-1)^n
\]
shows that $5\nmid L_n$ for odd $n$, so $5$ is invertible modulo every $w_j$.  For completeness, if $M,N$ are positive odd integers and $d=\gcd(M,N)$,
then $\gcd(L_M,L_N)=L_d$. Indeed, the Lucas addition identity
$L_{r+2N}=L_NL_{r+N}-(-1)^NL_r$ gives
$L_{r+2N}\equiv L_r\pmod{L_N}$ for odd $N$.
Reduce $M$ modulo $2N$ and, if necessary, replace the residue by its
negative to obtain an odd $R$ with $0<R\le N$.
Since $L_{-R}=-L_R$ for odd $R$, this preserves both
$\gcd(M,N)$ and $\gcd(L_M,L_N)$. The resulting odd-index Euclidean
algorithm ends at $d$, proving the formula. For odd $d$, the same addition identity gives
\[
  L_{r+2d}\equiv L_r\pmod{L_d}.
\]
Take $M=2m+1$ and $N=2n+1$.  Since $M/d$ and $N/d$ are odd, $M-N$ is divisible by $2d$, and therefore
\[
  w_{m+1}\equiv w_{n+1}\pmod{\gcd(w_m,w_n)}.
\]
After multiplication by the inverse of $5$, these are precisely the pairwise compatibility conditions for the generalized Chinese remainder theorem applied to \eqref{eq:app-lucas-crt}.
\end{proof}

\begin{proof}[Proof of Theorem~\ref{thm:persistent-saturation-regimes}\textup{(c)}]
The second identity in \eqref{eq:app-w-recurrence} gives
\[
  \frac{P(3w_j)}{3w_{j-1}}
  =3\frac{w_j^2-5}{w_{j-1}}
  =3w_{j+1},
\]
so the displayed points form an integer Vieta orbit on
\[
  x^2+y^2-45=3xy.
\]
The projective determinant of this conic is $225/4\ne0$, while the recurrence \eqref{eq:app-w-recurrence} has dominant root $(3+\sqrt5)/2>1$; hence the conic is nonsingular and irreducible and the orbit is unbounded and nonperiodic.

Choose $a$ from Lemma~\ref{lem:app-lucas-crt}.  Multiplying \eqref{eq:app-lucas-crt} by $w_{j-1}$ and using \eqref{eq:app-w-recurrence} gives
\[
  5a w_{j-1}\equiv w_j^2-5\equiv-5\pmod{w_j}.
\]
Since $5$ is invertible modulo $w_j$,
\begin{equation}\label{eq:app-hyperbolic-source}
  w_j\mid a w_{j-1}+1.
\end{equation}
For
\[
  \widetilde P_a(T)=P(T)+T(aT+3),
\]
the source condition at $(x_j,y_j)$ becomes
\[
  x_j=3w_j\mid ay_j+3=3(a w_{j-1}+1),
\]
which follows from \eqref{eq:app-hyperbolic-source}.  Quadratic common-constant reciprocal normalization gives
\[
  \widetilde C_a(T)=T^2+3T-45(1+a),
\]
so
\[
  \widetilde C_a(T)-P(T)=3(T-15a).
\]
Finally \eqref{eq:app-lucas-crt} yields
\[
  y'_j-15a
  =3(w_{j+1}-5a)
  \equiv0\pmod{3w_j}=0\pmod{x_j},
\]
which is genuine nonunit $3$-saturation because $3\mid x_j$.

It remains to rule out infinite persistence for a fixed $a$.  If it held for infinitely many $j$, then
\[
  k_j=\frac{w_{j+1}-5a}{w_j}\in\mathbb Z
\]
for infinitely many $j$.  Along that subsequence,
\[
  k_j\longrightarrow\frac{3+\sqrt5}{2},
\]
because $w_j\to\infty$.  A convergent sequence of integers is eventually constant, impossible because the limit is irrational.
\end{proof}

\begin{remark}[primitive auxiliary choice]\label{rem:hyperbolic-primitive-choice}
The solution $a$ of \eqref{eq:app-lucas-crt} may be chosen so that $\widetilde P_a$ is primitive.  Modulo $3$, recurrence \eqref{eq:app-w-recurrence} gives $w_{j+1}\equiv-w_{j-1}$, and $w_0\equiv w_1\equiv1$ yields the repeating pattern $1,1,2,2$. Thus no $w_j$ is divisible by $3$, and the common modulus of the CRT solution class is coprime to $3$.  Shifting $a$ within that class allows $a\not\equiv-1\pmod3$, and then $\gcd(1+a,3,45)=1$.
\end{remark}
\end{document}